\documentclass[11pt,english]{article}
 \usepackage[english]{babel}
\usepackage{cite} 
 \usepackage[dvips]{graphicx}

\usepackage[usenames, dvipsnames]{xcolor}
\usepackage[utf8x]{inputenc}
\usepackage[T1]{fontenc}
\usepackage{tikz}
\usepackage{pgfplots}

\usepackage[backref=page]{hyperref}

\usepackage{amsfonts,amsmath,amsxtra,amsthm,amssymb,latexsym}
\newtheorem{theorem}{Theorem}[section]
 
 \newtheorem{lemma}[theorem]{Lemma}
 \newtheorem{proposition}[theorem]{Proposition}
 \theoremstyle{definition}
 \newtheorem{definition}[theorem]{Definition}
 \theoremstyle{remark}
 \newtheorem{remark}[theorem]{Remark}
 
 \numberwithin{equation}{section}

\newcommand{\R}{\mathbb R}
\begin{author}
{Jaime Angulo Pava $^1$ and M\'arcio Cavalcante $^2$}
\end{author}

\begin{title}
{Linear instability of stationary solutions  for the Korteweg-de Vries equation on a star graph}
\end{title}
\date{}
\begin{document}
\maketitle

\centerline{$^1$ Department of Mathematics,
IME-USP}
 \centerline{Rua do Mat\~ao 1010, Cidade Universit\'aria, CEP 05508-090,
 S\~ao Paulo, SP, Brazil.}
 \centerline{\it angulo@ime.usp.br}
 
  \centerline{ $^2$ Institute of Mathematics,  Universidade Federal de Alagoas,
Macei\'o, AL, Brazil.}
 \centerline{\it  marcio.melo@im.ufal.br }

\begin{abstract}
The aim of this work is to establish a linear instability criterium of stationary solutions for the Korteweg-de Vries model on a  star graph with a structure represented by a finite collections of semi-infinite edges. By considering a boundary condition of $\delta$-type interaction at the graph-vertex,   we show that the continuous  tail and bump profiles are linearly unstable in a balanced star graph. The use of the analytic perturbation theory of operators and the extension theory of symmetric operators is a piece fundamental in our stability analysis.

The  arguments presented in this investigation
has prospects for the study of the instability of
stationary waves solutions  of other nonlinear evolution equations on star graphs.

\end{abstract}

\textbf{Key  words}. Korteweg-de Vries model, star graph, tail, bump, $\delta$-type interaction, perturbation theory, extension theory, instability.

\section{Introduction}

A quantum graph is a metric graph, i.e., a network-shaped structure of vertices connected by edges, with a linear Hamiltonian operator (such as a Schr\"odinger-like operator or the  Airy-like operator) suitably defined on functions that are supported on the edges. It arises as a simplified models in 
for wave propagation, for instance, in a  quasi one-dimensional 
(e.g. meso- or nanoscale) system that looks like a thin 
neighborhood of a  graph. Quantum graph have been used to describe a variety of physical problems and applications, such as in chemistry and engineering (see \cite{BK, BlaExn08, BurCas01, K, Mug15} for details and references). Recently, they have attracted much attention in the context of soliton transport in networks and branched structures (see \cite{SBM, SMSSK}) since wave dynamics in networks can be modeled by nonlinear evolution equations suitably defined on the edges. Soliton and other nonlinear waves in branched systems appear in different system of condensed matter, Josephson junction networks, polymers, optics, neuroscience, DNA, blood pressure waves in large arteries or in shallow water equation to describe a fluid network (see \cite{AdaNoj13a, Berkolaiko, BK, BeK, BurCas01, CM, Fid15, K,  Mehmeti, Mug15, Noj14} and references therein). To address these issues, in general the problem is difficult to tackle because both the equation
of motion and the geometry are complex. A first direction is to look at what happens in a simpler geometry, such as  a $\mathcal Y$-junction framework, and to examine the linear equation associated to the  nonlinear model. But, in many cases however the nonlinearity can not be neglected, by instance, in fluid system to describe a fluid network.

Thus, in the last years the study of nonlinear dispersive models
on metric graph has attracted a lot of attention of mathematician 
and physicists. In particular, the prototype of framework (graph-geometry)  for
description of these phenomena have been a {\it star graph}
$\mathcal G$, namely, a metric graph with $N$ 
half-lines of the form $(0, +\infty)$  connecting
at a common vertex $\nu=0$, together with a nonlinear equation suitably defined on the edges such as the nonlinear Schr\"odinger
equation (see Adami {\it{et al.}} \cite{AdaNoj14, AdaNoj15},  Angulo {\it{et al.}}  \cite{AngGol17a, AngGol17b}),  or the BBM equation (see Bona  {\it{et al.}} \cite{bona} and Mugnolo  {\it{et al.}} \cite{Mugnolobbm}). We note that with the introduction of the nonlinearity in the dispersive model, the network provides a nice field where one can looking for interesting soliton propagations and nonlinear dynamics in general. However, there are few exact analytic study of soliton propagation through networks by the nonlinear flow induced by the equation.  Results on the stability or instability mechanism of these profiles are still unclear.  One of the objectives of this work is to shed light on these themes.
A central point that makes this analysis a delicate problem is the presence of a vertex where the underlying one-dimensional star graph should bifurcate (or multi-bifurcate in a general metric graph). We note that not branching angles but the topology of bifurcation is essential. Indeed,  a soliton-profile coming into the vertex along one of the bonds shows a complicated motion around the vertex
such as reflection and emergence of the radiation there, moreover, in particular one cannot see easily how energy travels across the network. Therefore,  the study of the dynamic for non-linear evolution models becomes a challenge and it will depend  heavily on the conditions on the vertex (or vertices) to have a fruitful description of the system . 

In this work, we consider the well-known Korteweg-de Vries (KdV) equation in context of a metric  graph $\mathcal G$,  it which will have  a structure represented by  finite  collections of semi-infinite edges  parametrized by $(-\infty, 0)$ or $(0, +\infty)$.  In this case, $\mathcal G$  is sometimes also called a star-shaped metric  graph (see Figure \ref{figure2}). 

 We recall that the KdV equation on all the line 
$$
\partial_t u + \partial_x^3u + \partial_x u+ u\partial_xu = 0, \quad u=u(x,t),\;x\in \mathbb R, 
$$ 
it was first derived by Korteweg and de-Vries \cite{KDV} in 1895 as a model for long waves propagating
on a shallow water surface. Recently, the KdV equation have been appearing in other context. More precisely, this equation has been used  as a model to study blood
pressure waves in large arteries. In this way, for example, Chuiko and  Dvornik  in \cite{Chuiko} proposed a new computer model for systolic pulse waves within the cardiovascular system based on the KdV equation. Also, Cr\'epeau and  Sorine  in \cite{Crepeau} showed that some particular solutions of the KdV equation, more exactly, the 2 and 3-soliton well-known solutions, seem to
be good candidates to match the observed pressure pulse waves.
\vskip0.2in
\begin{figure}[htp]\label{figure2}
	\centering 
	\begin{tikzpicture}[scale=3]
	\draw [thick, dashed] (-1.3,0)--(-1,0);
	\draw[thick](-1,0)--(0,0);
	\node at (-0.65,0.1){$(-\infty,0)$};
	
		\draw [thick, dashed] (1.3,0)--(1,0);
	\draw[thick](1,0)--(0,0);
	\node at (0.65,0.1){$(0,+\infty)$};
	
	\draw[thick](0,0)--(-0.89,0.45);
	\node at (-0.46,0.33)[rotate=-3
	0]{$(-\infty,0)$};
	\draw [thick, dashed] (-0.89,0.45)--(-1.19,0.6);
	
	\draw[thick](0,0)--(-0.89,-0.45);
	\node at (-0.46,-0.33)[rotate=30]{$(-\infty,0)$};
	\draw [thick, dashed] (-0.89,-0.45)--(-1.19,-0.6);
	\fill (0,0)  circle[radius=1pt];

	\draw[thick](0,0)--(0.89,0.45);
	\node at (0.46,0.33)[rotate=3
	0]{$(0,+\infty)$};
	\draw [thick, dashed] (0.89,0.45)--(1.19,0.6);
	\draw[thick](0,0)--(0.89,-0.45);
	\node at (0.46,-0.33)[rotate=-30]{$(0,+\infty)$};
	\draw [thick, dashed] (0.89,-0.45)--(1.19,-0.6);
	\fill (0,0)  circle[radius=1pt];
	\end{tikzpicture}
	\caption{A star-shaped metric graph with $6$ edges}
\end{figure}
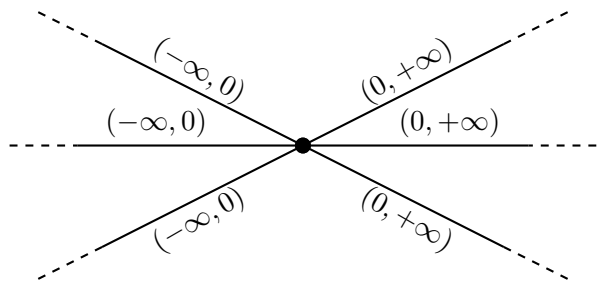

\vskip0.2in

In the mathematical context, the Cauchy problem for the KdV posed on all the line, torus, on the half-lines and on a finite interval have been well studied in the last years, we refer as an example \cite{Bonaint, CK, Faminskii, Guo, Holmer, Jia, KPV, Kishimoto} and references therein.  We also notice the recent result of Cavalcante and Mu\~noz \cite{CM2} about  that solitons
posed initially far away from the origin are strongly stable for the IBVP associated to the
KdV equation posed on the right half-line, assuming homogeneous boundary conditions.

Studies for the linearized Korteweg-de Vries equation on star-shaped metric graphs have started appearing recently. In
\cite{Sobirov1} was studied existence and uniqueness of solutions for the linearized KdV equation on metric star graphs by using potential theory, where the solutions were obtained
in the class of Schwartz and in Sobolev classes with high order. Very recently, Mugnolo, Noja and Seifert \cite{MNS} obtained a characterization of all boundary conditions under which the
Airy-type evolution equation 
\begin{equation}\label{kdv0}
\partial_{t}u_\bold e(x,t)=\alpha_ \bold e \partial_x ^3u_\bold e(x,t) + \beta_ \bold e \partial_x u_\bold e(x,t), \quad x\neq 0, t\in \mathbb{R}, \bold e\in \bold E,
\end{equation}
generates either a semigroup or a group on a metric star graph $\mathcal G$ with a structure represented by the set $
\bold E\equiv \bold E_{-}\cup \bold E_{+}$
where $\bold E_{+}$ and $\bold E_{-}$ are finite or countable collections of semi-infinite edges $\bold e$ parametrized by $(-\infty, 0)$ or $(0, +\infty)$, respectively. The half-lines are connected at a unique vertex $\nu=0$. Here  $(\alpha_ \bold e)_{\bold e\in \bold E}$ and  $(\beta_ \bold e)_{\bold e\in \bold E}$ are two sequences of real numbers.

As far as we know, the study of the nonlinear Korteweg-de Vries equation in star graphs  is relatively underdeveloped.   We notice the recent result of Cavalcante in \cite{Cav1} about the local well-posedness for the Cauchy problem associated to Korteweg-de
Vries equation on a metric star graph with three semi-infinite edges given by one negative half-line and two positives half-lines attached to a common vertex $\nu=0$ (the $\mathcal Y$-junction framework). Recent results of stabilization and boundary controllability for KdV equation on bounded star-shaped graphs was obtained by Ammari and  Crepeau \cite{Ammari} and  Cerpa, Crepeau and Moreno\cite{Cerpa}.

The focus of our study here will be the following vectorial KdV model 
\begin{equation}\label{kdv3}
\partial_{t}u_\bold e(x,t)=\alpha_ \bold e \partial_x ^3u_\bold e(x,t) + \beta_ \bold e \partial_xu_\bold e(x,t)+2 u_\bold e (x,t) \partial_x u_\bold e (x,t),\;\;\; \bold e\in \bold E=\bold E_{-}\cup \bold E_{+},
\end{equation}
on a star graph $\mathcal G$.  We are interested for the  first time (as far as we know) in the dynamics generated by the flow of the KdV model \eqref{kdv3} around solutions of stationary type
$$
(u_{\bold e}(x,t))_{\bold e\in \bold E}=(\phi_{\bold e}(x))_{\bold e\in \bold E}
$$
where for $\bold e\in \bold E_{-}$ the profile $\phi_{\bold e}: (-\infty, 0)\to \mathbb R$ satisfy $\phi_{\bold e}(-\infty)=0$, and for  $\bold e\in \bold E_{+}$ $\phi_{\bold e}: (0, \infty)\to \mathbb R$ satisfy $\phi_{\bold e}(+\infty)=0$. The existence of profiles of stationary type, namely, solutions of the following nonlinear elliptic equation
\begin{equation}\label{kdv3a}
\alpha_ \bold e  \frac{d^2}{dx^2} \phi_{\bold e}(x)+ \beta_\bold e \phi_{\bold e}(x)+\phi_{\bold e}^2(x)=0,\;\;\; \bold e\in \bold E,
\end{equation}
are well know and the profile depend of the soliton associated to the KdV on  the full line, 
\begin{equation}\label{kdv3b}
\phi_{\bold e}(x)=c(\alpha_ \bold e,  \beta_\bold e) sech ^2(d(\alpha_ \bold e,  \beta_\bold e)x +p_\bold e),\;\;\; \bold e\in \bold E.
\end{equation}
For instance,  for $\alpha_ \bold e>0$ and $0>\beta_\bold e$, for each $\bold e\in \bold E$, we can obtain different family of profiles satisfying the conditions $\phi_{\bold e}(\pm \infty)=0$, $\bold e\in \bold E_{\pm}$ (see section 3 below). The specific value  of the shift $p_\bold e$ will depend which other (or others) condition(s)  imposed on the profile $\phi_{\bold e}$  is determined on the vertex of the graph $\nu=0$. 

The  main interest of our study here with regard to the nonlinear model \eqref{kdv3}  is to establish a linear instability criterium for  stationary profiles  on a star graph $\mathcal G$. 

 A starting point  for the one previously described, it is to determine when the Airy type operator 
\begin{equation}\label{Air2}
A_0:  (u_ \bold e)_{\bold e\in \bold E}\to \Big (\alpha_ \bold e \frac{d^3}{dx^3}u_\bold e + \beta_ \bold e \frac{d}{dx} u_\bold e\Big)_{\bold e\in \bold E}
\end{equation}
being seen as an unbounded operator on a certain Hilbert space,  it will have extensions $A_{ext}$ on $L^2(\mathcal G)$ such that the dynamics induced by the linear evolution problem
\begin{equation}\label{evolu}
\left\{ \begin{array}{ll}
z_t=A_{ext} z,\\
z(0)= u_0\in D(A_{ext}),
  \end{array}  \right.
\end{equation}
it is given by a $C_0$-group, $z(t)=e^{tA_{ext}}u_0$. In this  point the theory in \cite{MNS} and \cite{SSVW} give us that properties of the induced dynamics can be obtained by studying boundary operators in the corresponding boundary space induced  by the vertex of the graph. Here we are interested when the extension $A_{ext}$ is a  skew-self-adjoint operator. So, from  Stone's Theorem we obtain that the  dynamics in \eqref{evolu} is given by a unitary group. Two delicate issues emerge at this point of the analysis and which are necessary in our study. The first one  is how to determine all the possible skew-self-adjoint extensions of   the Airy  operator $A_0$, and the second one to find some kind of formula for the unitary groups associated to these extensions. With regard to the first issue,  a characterization of all skew-self-adjoint extensions of $A_0$ was obtained recently by Mugnolo, Noja and Seifert in \cite{MNS} via Krein spaces (see also Schubert, Seifert, Voigt and Waurick in \cite{SSVW}). In Section 2.1 we give a brief description of this theory. In particular, we construct  in Proposition \ref{L}  below a family of  skew-self-adjoint extensions  $(A_Z, D(A_Z))$ of $\delta$-type interaction  for $A_0$ in the case of a star graph with two half-lines. Thus, we determine  which stationary solutions $ (\phi_ \bold e)_{\bold e\in \bold E}$ with $\phi_ \bold e$ defined in \eqref{kdv3b} belong to the domain $D(A_Z)$. In this form,  we found that the only possible profiles can  be of the type either tail or bump (see Figures 2 and 3 below). In section 6, we extend the construction of  skew-self-adjoint extensions   of $\delta$-type interaction  for $A_0$ on a general star graph $\mathcal G$ with  $|\bold E_{+}|=|\bold E_{-}|=n\geqq 2$ (balanced star graphs).

With regard to specific formulas  for  the unitary groups  associated to the possible skew-self-adjoint extensions of  the Airy  operator $A_0$, it is an open problem in general. Now from the extension theory (see Proposition \ref{ext}) we know that there must be $9|\bold E_{+}|^2$ family of unitary groups for the case of a star graph $\mathcal G$ with  $|\bold E_{+}|=|\bold E_{-}|=n\geqq 2$, each one having its own representation.  In Proposition \ref{group24} (Appendix) via Green functions, we establish by first time a formula for the unitary group associated to the  one-parameter family of  skew-self-adjoint extensions  of $\delta$-type  in    \eqref{L_Z} for $A_0$ on a balanced star graph $\mathcal G$ ($|\bold E_{+}|=|\bold E_{-}|$).

Now, in Theorem \ref{crit} (section 4)  we establish our linear instability criterium  for stationary solutions for the Korteweg-de Vries model \eqref{kdv3} on a star graph not necessarily balanced. This instability criterium can be seen as an extension of  Lopes's result in \cite{Lopes} (see also Grillakis, Shatah and Strauss  \cite{GrilSha87, GrilSha90} and Pego and Weinstein \cite{pego}). Theorem \ref{crit} will be applied to the  family  of stationary profiles of tail and bump type that  appear with vertex conditions of $\delta$-type, and we  obtain that they are linearly unstable when $|\bold E_{+}|=|\bold E_{-}|=n\geqq 1$ (see Figures 2-3-4-5 and Theorems \ref{main} and \ref{main2}).  In the case $n=1$, linear instability analysis  is based in the   analytic perturbations theory of operators, while the case of $n\geqq 2$, analytic perturbation and the  extension theory of symmetric operators of Krein and von Neumann are required. We have divided our stability study into two cases ($n=1$ and $n\geqq 2$ separately) to make it clear how the geometry of the graph induces an addition of new tools in the analysis.

The existence and stability of other  families of stationary profiles for the KdV model \eqref{kdv3} defined on a different graph-geometry (non-balanced graphs) is being the goal of a work in progress. As well as, a stability study for the generalized KdV model
\begin{equation}
\partial_{t}u_\bold e=\alpha_ \bold e \partial_x ^3u_\bold e + \beta_ \bold e \partial_xu_\bold e+p u^{p-1}_\bold e  \partial_x u_\bold e ,\;\;\; \bold e\in \bold E,
\end{equation}
and $p\in \mathbb N$, $p\geqq 2$.

The paper is organized as follows. In the Preliminaries (Section 2) we give some brief
description of the existence of unitary groups for Airy operators via extension theory and examples in the case of boundary conditions of $\delta$-type at the vertex $\nu=0$ for  two half-lines. The existence of stationary solutions of tail and bump type is given in Section 3. Our linear instability criterium on a general star graph  is established in Section 4. Section 5 and 6 are dedicated to establish our main results of linear instability of tail and bump profile for the KdV model \eqref{kdv3}. In Appendix we briefly discuss some tools of the extension theory of Krein and von Neumann used in our study of linear instability.  Also, we give a  unitary group representation associated to the  one-parameter family of  skew-self-adjoint extensions  $(A_Z, D(A_Z))$  defined in \eqref{domain8} and 
  $(H_Z, D(H_Z))$ defined in \eqref{L_Z}.

\vskip0.2in

\noindent \textbf{Notation.} Let $-\infty\leq a<b\leq\infty$. We denote by $L^2(a,b)$  the  Hilbert space equipped with the inner product $(u,v)=\int\limits_a^b u(x)\overline{v(x)}dx$.
 By $H^n(\Omega)$  we denote the classical  Sobolev spaces on $\Omega\subset \mathbb R$ with the usual norm.   We denote by  $\mathcal{G}$ the star graph parametrized by
  $\bold E=\bold E_{-}\cup \bold E_{+}$, where  $\bold E_{-}$ and $\bold E_{+}$ are sets of half-lines of the form  $(-\infty, 0)$ and   $(0, +\infty)$, respectively,  attached to a common vertex $\nu=0$. On the graph we define the classical spaces 
  \begin{equation*}
  L^p(\mathcal{G})=\bigoplus\limits_{\bold e\in  \bold E_{-} }L^p(-\infty, 0)  \oplus \bigoplus\limits_{\bold e\in  \bold E_{+} }L^p(0, +\infty), \quad \,p>1,
  \end{equation*}   
 and 
  \begin{equation*}  
 \quad H^n(\mathcal{G})=\bigoplus\limits_{\bold e\in  \bold E_{-} }H^n(-\infty, 0) \oplus \bigoplus\limits_{\bold e\in  \bold E_{+} } H^n(0, +\infty), 
 \end{equation*}   
with the natural norms. Also, for $u= (u_ \bold e)_{\bold e\in \bold E}, v= (v_ \bold e)_{\bold e\in \bold E}\in L^2(\mathcal G)$, the inner product is defined by
$$
\langle u, v\rangle= \sum _{\bold e\in \bold E_-}\int_{-\infty}^0 u_ \bold e\overline{v_ \bold e}dx +\sum _{\bold e\in \bold E_+}\int_0^{\infty} u_ \bold e\overline{v_ \bold e}dx. 
$$
We also denote sometimes $(u_ \bold e)_{\bold e\in \bold E}$, as $(u_ \bold e)_{\bold e\in \bold E}=(u_{1, -}, ...,u_{m, -}, u_{1, +},...,u_{n, +})$. Depending on the context we will use the following notations for different objects. By $||\cdot||$ we denote  the norm in $L^2(\Omega)$ ($\Omega=(-\infty, 0)$ or $(0, +\infty)$)) or in $L^2(\mathcal{G})$. By $||\cdot||_p$ we denote  the norm in $L^p(\Omega)$  or in $L^p(\mathcal{G})$. By depending of the context  we identify  $u=(u_{-}, u_{+})\in  L^2(\mathcal{G})$ as a element in $\Pi_{i=1}^m L^2(-\infty, 0)\times \Pi_{i=1}^n L^2(0, +\infty)$, with $m=|\bold E_{-}|$ and $n=|\bold E_{+}|$ or as $(m+n)\times 1$-matrix column.
  
 Let $A$ be a  closed densely defined symmetric operator in the Hilbert space $H$. The domain of $A$ is denoted by $D(A)$. The deficiency indices of $A$ are denoted by  $n_\pm(A):=\dim Ker(A^*\mp iI)$, with $A^*$ denoting the adjoint operator of $A$.  The number of negative eigenvalues counting multiplicities (or Morse index) of $A$ is denoted by  $n(A)$.

\section{Preliminaries}

By convenience of the reader, we give some brief description about the characterization of all skew-self-adjoint extensions of the Airy operators associated to \eqref{kdv0}. Our strategy will follow the theory recently established  by Mugnolo, Noja and Seifert in \cite{MNS}.

\subsection{Airy operators and the existence of unitary groups} 

In this subsection, we will define properly for sequences of real numbers $(\alpha_ \bold e)_{\bold e\in \bold E}$ and  $(\beta_ \bold e)_{\bold e\in \bold E}$, the following Airy operator
\begin{equation}\label{Air}
A_0:  (u_ \bold e)_{\bold e\in \bold E}\mapsto \Big (\alpha_ \bold e \frac{d^3}{dx^3}u_\bold e + \beta_ \bold e \frac{d}{dx} u_\bold e\Big)_{\bold e\in \bold E}
\end{equation}
as an unbounded operator on a certain Hilbert space, in such a way that the possible extensions $A_{ext}$ induce that the solution of the linear  equation
\begin{equation}\label{Air1}
z_t=A_{ext} z,
\end{equation}
 it is given by a  $C_0$-unitary group, in other words, by the Stone's theorem we need that $A_0$ has skew-self-adjoint extensions $A_{ext}$, on $L^2(\mathcal G)$. Since the Airy operator $A_0$ is of odd order, changing the sign of each constant $\alpha_ \bold e$ it is equivalent to exchange the positive and negative half-lines and so we can choose $\alpha_ \bold e>0$ for every $\bold e\in \bold E$ without loss of generality. The following proposition from Mugnolo, Noja and Seifert \cite{MNS} give us an answer about the problem associated to \eqref{Air1}.
 
 \begin{proposition} \label{ext} Let $\mathcal G$ be a star graph determined by $
\bold E\equiv \bold E_{-}\cup \bold E_{+}$ and let  $(\alpha_ \bold e)_{\bold e\in \bold E}$, $(\beta_ \bold e)_{\bold e\in \bold E}$ be two sequences of real numbers with $\alpha_ \bold e>0$ for all $\bold e \in \bold E$.  Consider the operator $A_0$ defined in \eqref{Air} with
$$
D(A_0)\equiv \bigoplus\limits_{\bold e\in  \bold E_{-} }C_c ^{\infty}(-\infty, 0)  \oplus \bigoplus\limits_{\bold e\in  \bold E_{+} }C_c ^{\infty}(0, +\infty).
 $$
Then, $iA_0$ is a densely defined  symmetric operator on  the Hilbert space
$$
 L^2(\mathcal{G})=\bigoplus\limits_{\bold e\in  \bold E_{-} }L^2(-\infty, 0)  \oplus \bigoplus\limits_{\bold e\in  \bold E_{+} }L^2(0, +\infty),
$$
with deficiency indices $(n_+(iA_0), n_-(iA_0))= (2| \bold E_{-}|+| \bold E_{+}|, | \bold E_{-}|+2| \bold E_{+}|)$. Therefore, $A_0$ has skew-self-adjoint extension on $ L^2(\mathcal{G})$ if and only if $| \bold E_{-}|=| \bold E_{+}|$.
 \end{proposition}

For $| \bold E_{-}|=| \bold E_{+}|$, {\it{i.e.}} the number of incoming half-lines is the same of outgoing half-lines, the graph $\mathcal G$ is called {\it balanced}.

Some comments about the former proposition deserve to be made which will be very useful in our study.
\begin{remark}\label{extension}
 From Proposition \ref{ext} and  from the classical Krein-von Neumann  extension theory for symmetric operators (see Chapter 4 in Naimark \cite{Nai67} and Theorem X.2 in Reed and Simon \cite{RS}) the operator $(A_0, D(A_0))$, on the case of balanced star graphs, admits a $9| \bold E_{+}|^2$-parameter family of skew-self-adjoint extension generating each one a unitary dynamics on $L^2(\mathcal{G})$ associated to the linear evolution equation \eqref{Air1}. Moreover, every skew-self-adjoint extension $(A, D(A))$  is obtained as a restriction of $(-A^*_0, D(A^*_0))$ with  $-A_0^*=A_0$ and
\begin{equation}\label{domain}
D(A_0^*)\equiv \bigoplus\limits_{\bold e\in  \bold E_{-} }H^3(-\infty, 0)  \oplus \bigoplus\limits_{\bold e\in  \bold E_{+} }H^3(0, +\infty).
\end{equation}
Moreover, we can see the action of $A_0$ as being a matrix-diagonal operator 
$$
A_0=\text{diag} \Big(\Big (\alpha_ \bold e \frac{d^3}{dx^3}u_\bold e + \beta_ \bold e \frac{d}{dx} u_\bold e\Big)\delta_{ij}\Big),\quad 1\leqq i,j\leqq | \bold E_{+}|+ | \bold E_{-}|. 
$$
\end{remark}

We empathize that, the complete characterization of all skew-self-adjoint extensions of $(A_0, D(A_0))$ is a bit complex and one strategy for finding these was obtained very recently by Mugnolo, Noja and Seifert in \cite{MNS} via Krein spaces (see also Schubert, Seifert, Voigt and Waurick in \cite{SSVW}). The central idea of the process is given in Theorem 3.7 and Theorem 3.8 in \cite{MNS} where skew-self-adjoint extensions are parametrized through relations between boundary values.  Here, we will use this approach and  for convenience of the reader we briefly explain this one for a balanced star graph $\mathcal{G}$. For abbreviating our notations, for $u=(u_ \bold e)_{\bold e\in \bold E}\in D(A_0^*)$ in \eqref{domain} we denote 
$$
u(0-)\equiv(u_ \bold e(0-))_{\bold e\in \bold E_{-}},\;\;\text{and}\;\; u(0+)\equiv(u_ \bold e(0+))_{\bold e\in \bold E_{+}}
$$
and  we consider the space of boundary values in $\mathbb C^{3n}$,  with $n=|\bold E_{\pm}|$,
$$
(u(0-), u'(0-), u''(0-)),\;\;and\;\;(u(0+), u'(0+), u''(0+)),
$$
 spanning respectively subspaces $\mathbb G_{-}$ and $\mathbb G_{+}$ in $\mathbb C^{3n}$. Next, the boundary form of the operator $A_0$ is easily seen  for $u, v \in D(A_0^*)$ to be (where we are identifying a vector   with its transpose)
\begin{equation}\label{domain3}
\begin{split}
\langle A_0^*u, v\rangle+&\langle u, A_0^*v\rangle\\
&=
\left(B_{-}\left(\begin{array}{c} u(0-)\\u'(0-) \\u''(0-)\end{array}\right), \left(\begin{array}{c} v(0-)\\v'(0-) \\v'(0-)\end{array}\right)\right)_{\mathbb G_{-}}- \left(B_{+}\left(\begin{array}{c} u(0+)\\u'(0+) \\u''(0+)\end{array}\right), \left(\begin{array}{c} v(0+)\\v'(0+) \\v'(0+)\end{array}\right)\right)_{\mathbb G_{+}} 
\end{split}
\end{equation}
where for $I=I_{n\times n}$ representing the identity matrix of order $n\times n$, we have
\begin{equation}\label{domain4}
B_{-}=\left(\begin{array}{ccc}-  I\beta_{-} & 0 & -I\alpha_{-}\\0 &  I \alpha_{-}& 0 \\- I\alpha_{-}& 0 & 0\end{array}\right),\quad
B_{+}=\left(\begin{array}{ccc}-I\beta_{+} & 0 & -I\alpha_{+} \\0 & I\alpha_{+} & 0 \\- I \alpha_{+} & 0 & 0\end{array}\right)
\end{equation}
and $\alpha_{\pm} =(\alpha_ \bold e)_{\bold e\in \bold E_{\pm}}$, $\beta_{\pm} =(\beta_ \bold e)_{\bold e\in \bold E_{\pm}}$. Thus, by considering the (indefinite) inner product $\langle \cdot | \cdot\rangle_{\pm}: \mathbb G_{\pm}\times \mathbb G_{\pm}\to \mathbb C$ by
$$
\langle x | y\rangle_{\pm}\equiv ( B_{\pm} x , y)_{\mathbb G_{\pm}},\qquad x, y\in \mathbb G_{\pm}
$$ 
we obtain that $(\mathbb G_{\pm}, \langle \cdot | \cdot\rangle_{\pm})$ are Krein spaces and $\langle \cdot | \cdot\rangle_{\pm}$ is non-degenerate (for $x\in \mathbb G_{\pm}$ with $\langle x | x\rangle_{\pm}=0$ follows $x=0$). Thus, from Theorem 3.8 in \cite{MNS} we have that for a linear operator $L: \mathbb G_{-}\to \mathbb G_{+}$, the operator $(A_L, D(A_L))$ defined by  
\begin{equation}\label{domain5}
\left\{ \begin{array}{ll}
&A_Lu=-A_0^*u=A_0u, \\
& D(A_L)=\Big\{u\in  D(A_0^*): L(u(0-), u'(0-), u''(0-))=(u(0+), u'(0+), u''(0+))\Big \} \\
  \end{array}  \right.
\end{equation}
it is a skew-self-adjoint extension of $(A_0, D(A_0))$ if and only if $L$ is $(\mathbb G_{-}, \mathbb G_{+})$-unitary, namely, 
\begin{equation}\label{domain6}
\langle Lx | Ly\rangle_+= ( B_{+} Lx , Ly)_{\mathbb G_{+}}= \langle x | y\rangle_{-}=( B_{-} x , y)_{\mathbb G_{-}},
\end{equation}
in other words, $L^*B_{+} L=B_{-}$. Indeed, 
for $u, v \in D(A_L)$  we get  from \eqref{domain3}  
\begin{equation*}
\begin{aligned}
&\langle -A_Lu, v\rangle+\langle u, -A_Lv\rangle=\langle A_0^*u, v\rangle+\langle u, A_0^*v\rangle\\
&=\langle u(0-)| v(0-)\rangle_{-}-\langle u(0+)| v(+)\rangle_{+}=\langle u(0-)| v(0-)\rangle_{-}-\langle Lu(0-)| Lv(-)\rangle_{+}.
\end{aligned}
\end{equation*}
Then, $(A_L)^*=-A_L$ if and only $L$ is $(\mathbb G_{-}, \mathbb G_{+})$-unitary.

Next, we consider the following family of skew-self-adjoint extension of $(A_0, D(A_0))$ in the case of two half-lines with a singular $\delta$-type interaction at the origin. Since  $|\bold E_{-}|=|\bold E_{+}|=1$ follows that $(A_0, D(A_0))$ admit a nine-parameter family of skew-self-adjoint extensions. Moreover, for  $u=(u_{-}, u_{+})\in H^3(-\infty, 0)  \oplus H^3(0,+\infty)$  we have that the subspaces $\mathbb G_{-}$ and $\mathbb G_{+}$ are given by the triplets $(u_{-}(0-), u'_{-}(0-), u_{-}''(0-))\subset  \mathbb C^3$ and $(u_{+}(0+), u_{+}'(0+), u''_{+}(0+))\subset \mathbb C^3$. Thus we have the following result.

\begin{proposition} \label{L} For $Z\in \mathbb R\setminus\{0\}$ we  define the linear operator $L_Z: \mathbb G_{-}\to \mathbb G_{+}$ by
\begin{equation}\label{domain7}
L_Z=\left(\begin{array}{ccc}1 & 0 & 0 \\Z & 1 & 0 \\ \frac{Z^2}{2} & Z & 1\end{array}\right).
\end{equation}
Then we obtain a family $(A_Z, D(A_Z))$ of skew-self-adjoint extensions of $(A_0, D(A_0))$ parametrized by $Z$ and which  are defined by
\begin{equation}\label{domain8}
\left\{ \begin{array}{ll}
                       &A_Zu= A_0u,\\
                       \\
  & D(A_Z) = \{u=(u_{-},  u_{+})\in H^3(-\infty, 0)  \oplus H^3(0,+\infty): u_{-}(0-)=u_{+}(0+),\\
  \\
  &\hskip0.7in u'_{+}(0+)- u'_{-}(0-)=Zu_{-}(0-),\; \frac{Z^2}{2}u_{-}(0-)+Zu'_{-}(0-)=u''_{+}(0+)-u''_{-}(0-)\}.
  \end{array}  \right.
\end{equation}
Moreover,  for $\alpha_ \bold e=(\alpha_{-}, \alpha_{+})\in \mathbb R^+\times \mathbb R^+ $ and $\beta_ \bold e =(\beta_{-}, \beta_{+})\in \mathbb R\times \mathbb R$ we need to have $\alpha_{-}= \alpha_{+}$ and $\beta_{-}=\beta_{+}$. We obtain that each element in $D(A_Z)$ can be seen as an element in $H^1(\mathbb R)$.
\end{proposition}

\begin{proof} From the extension theory framework established above, we see from \eqref{domain6}  that $L^*B_{+}L=B_{-}$ if and only if $\alpha_{-}=\alpha_{+}$ and $\beta_{-}=\beta_{+}$. Then the operator $A_Zu\equiv A_0 u$, $Z\in \mathbb R\setminus\{0\}$, defined for $u=(u_{-},  u_{+})$ such that 
$$
L_Z(u_{-}(0-), u'_{-}(0-), u_{-}''(0-))=(u_{+}(0+), u_{+}'(0+), u''_{+}(0+))
$$
 will represent a skew-self-adjoint extension family of $(A_0, D(A_0))$. This finishes the proof.
\end{proof}

\section{Stationary solutions in the case of two half-lines} 

In this section, we show the existence of  stationary solutions for the KdV model \eqref{kdv3}
on a star graph $\mathcal G$ represented by $ \bold E=(-\infty, 0)\cup (0,\infty)$ attaching at a common vertex $\nu=0$. As we will see, the  possibility of different profiles (modulo sign, etc.) is very varied. The choice of which profile may be viable for a possible study of its stability properties by the KdV flow depends on which domain of a skew-self-adjoint extension of $A_0$ this profile may come to belong. 

Thus, we consider the following stationary profile for the KdV model \eqref{kdv3},
\begin{equation}\label{kdv4a}
 (u_ \bold e(x,t))_{\bold e\in \bold E}=(\phi_ \bold e(x))_{\bold e\in \bold E}=(\phi_{-}(x), \phi_{+}(x))\in  H^3(-\infty, 0)  \oplus H^3(0,+\infty) 
\end{equation}
for $t\in \mathbb R$. Then, by substituting this profile  in \eqref{kdv3} and integrating once we obtain the following nonlinear system of elliptic equations
\begin{equation}\label{soli1}
\left\{ \begin{array}{ll}
                       & \alpha_{+}\phi''_{+}(x)+\beta_{+} \phi_{+}(x)+ \phi^2_{+}(x)=0,\quad x>0\\
                       \\
                       & \alpha_{-}\phi''_{-}(x)+\beta_{-} \phi_{-}(x)+ \phi^2_{-}(x)=0,\quad x<0.
  \end{array}  \right.
\end{equation}
Next, it is well know that the equation $a\psi''(x)+b \psi(x) +\psi^2(x)=0$ for all $x\in \mathbb R$ has nontrivial solutions for $\psi(\pm \infty)=0$ in the cases  either $a>0$ and $b<0$ or $a<0$ and $b>0$.  The first case represents the classical positive-soliton for the KdV model \eqref{kdv3}
\begin{equation}\label{soli2}
\psi(x)=-\frac{3b}{2} sech^2\Big(\frac12\sqrt{-\frac{b}{a}}\; x+p\Big), \;\;x\in \mathbb R, \quad p\in \mathbb R.
\end{equation}
The second case produces a depression soliton ($-\psi$ modulo translation). Next, we establish some specific profiles for $\phi_{\pm}$ (the first one will be the focus of our stability study here):

\begin{enumerate}
\item[1)] for $\alpha_{\pm}>0$ and $0>\beta_{\pm}$ or $\alpha_{\pm}<0$ and $0<\beta_{\pm}$:
\begin{equation}\label{soli3}
                       \phi_{\pm}(x)=-\frac{3}{2}\beta_{\pm} sech^2\Big(\frac12\sqrt{-\frac{\beta_{\pm}}{\alpha_{\pm}}}\; x+p_{\pm}\Big)
                       ,\end{equation}
the shift parameters $p_{\pm}$ depend  on boundary conditions for $\phi_{\pm}$ in the vertex-graph $\nu=0$.



\item[2)] For $\alpha_{+}<0$, $0<\beta_{+}$, and, $\alpha_{-}>0$, $0>\beta_{-}$: we obtain a combination of positive and negative non-continuous soliton profiles
\begin{equation}\label{soli5}
\left\{ \begin{array}{ll}
                       &\phi_{+}(x)=-\frac{3}{2}\beta_{+} sech^2\Big(\frac12\sqrt{-\frac{\beta_{+}}{\alpha_{+}}}\; x+p_{+}\Big), \;\; x>0,\\
 &\phi_{-}(x)=-\frac{3}{2}\beta_{-}sech^2\Big(\frac12\sqrt{-\frac{\beta_{-}}{\alpha_{-}}}\; x+p_{-}\Big), \;\; x<0.                      
                      \end{array}  \right.
                        \end{equation}

\end{enumerate}

Now, it which of the different profiles given above or other ones may be plausible to be studied by the dynamic of the 
 KdV equation \eqref{kdv3} on a specific graph will depend heavily on  the boundary conditions at the vertex $\nu=0$.
The following subsection give us an example of this situation, and also we show the rich variety of stationary profiles that may emerge from the KdV model on metric star graphs.

\subsection{Existence of stationary  solutions for a $\delta$-type interaction on two half-lines}

Our first example of solutions for \eqref{kdv3} will belong to the family of   skew-self-adjoint extension of $A_0$ via the operator $L_Z$ in \eqref{domain7}. Thus, from Proposition \ref{L} we obtain that for $\phi=(\phi_{-}, \phi_{+})$ in  $D(A_Z)$,  we need to have $\alpha_{+}=\alpha_{-}>0$, $\beta_{+}=\beta_{-}<0$, and so  the profiles  $\phi_{\pm}$  satisfy the same equation in \eqref{soli1} and from \eqref{soli3} for  $-\frac{\beta_{+}}{\alpha_{+}}>\frac{Z^2}{4}$ we obtain
\begin{equation}\label{soli6}
\phi_{+}(x)=-\frac{3\beta_{+}}{2} sech^2\Big(\frac{\sqrt{-\beta_{+}}}{2 \sqrt{\alpha_{+}}}\; x-tanh^{-1}\Big(\frac{Z \sqrt{\alpha_{+}}}{2\sqrt{-\beta_{+}}}\Big)\Big),\;\;x>0
\end{equation}
and $\phi_{-}(x)\equiv \phi_{+}(-x)$ for $x<0$. Since $\phi_{-}(0-)=\phi_{+}(0+)$ (continuity in zero), we note the condition
 \begin{equation}\label{condition}
 \phi''_{+}(0+)-\phi''_{-}(0-)=\frac{Z^2}{2}\phi_{-}(0-)+Z\phi'_{-}(0-)\end{equation} in \eqref{domain8}  is satisfied immediately. Figures 2-3 below show the profiles of $\phi_{\pm}$ for $Z\neq 0$. For $Z<0$, the so-called {\it tail} profile on the all line, and for $Z>0$, the so-called {\it bump} profile on the all line. Moreover, it is not difficult to show that the only stationary solutions (modulo sign) in $D(A_Z)$   from the KdV model \eqref{kdv3} are exactly the tail and bump profiles defined by formula \eqref{soli6}.
\begin{figure}[h]
\centering
\begin{minipage}[b]{0.34\linewidth}\label{figng22}
\includegraphics[angle=0,scale=0.5]{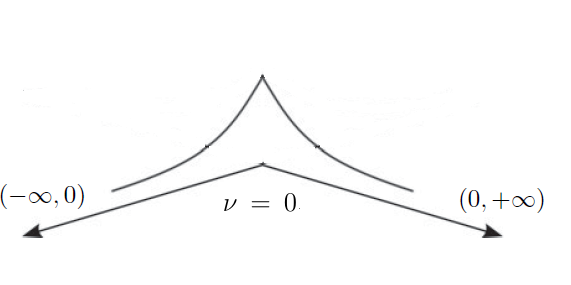}
\caption{$(\phi_{-}, \phi_{+})$ for $Z<0$}
\end{minipage}
\begin{minipage}[b]{0.34\linewidth}\label{fig33}
\includegraphics[angle=0,scale=0.5]{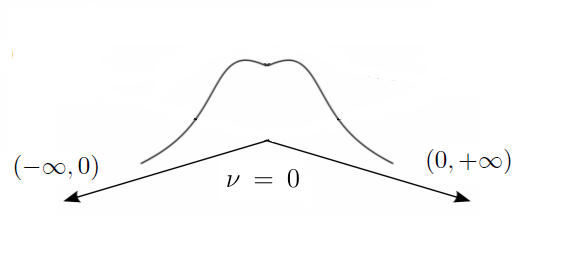}
\caption{$(\phi_{-}, \phi_{+})$ for $Z>0$}
\end{minipage}
\end{figure}

\section{Linear instability criterium for KdV on a start graph} 

 In this section, we establish a novel linear instability criterium of stationary solutions for the KdV model \eqref{kdv3} on a  start graph $\mathcal G$ with $|\bold E_{+}|=n$ and $|\bold E_{-}|=m$.  Thus, we will consider a extension $(A_{ext}, D(A_{ext}))$ of the Airy operator $A_0$ in \eqref{Air2} on $L^2(\mathcal G)$ such that the dynamic induced by the linear evolution problem \eqref{evolu} is given by a $C_0$-group (see \cite{MNS}).

 Suppose for $(\phi_{\bold e})_{\bold e\in \bold E}\in D(A_{ext})$ we have  that $(\tilde{u}_{\bold e}(x,t))_{\bold e\in \bold E}=(\phi_{\bold e}(x))_{\bold e\in \bold E}$ is a nontrivial  solution of \eqref{kdv3}, thus we obtain the following set  of $|\bold E_{-}|+|\bold E_{+}|$ equations
 
\begin{equation}\label{stat1}
\alpha_ \bold e \frac{d^3}{dx^3} \phi_{\bold e} + \beta_ \bold e \frac{d}{dx} \phi_{\bold e}+ 2\phi_{\bold e}\phi'_{\bold e}=0,\quad \bold e\in \bold E=\bold E_{-}\cup \bold E_{+}.
 \end{equation}
Then, since $\phi_{\bold e} (\pm \infty)=0$ we obtain for  $\bold e\in \bold E$ that each component satisfies the elliptic equation
 \begin{equation}\label{stat2}
\alpha_ \bold e \frac{d^2}{dx^2} \phi_{\bold e} + \beta_ \bold e \phi_{\bold e}+ \phi^2_{\bold e}=0.
 \end{equation}

 Next, we suppose  for $\bold e\in \bold E$, that $u_{\bold e}$ satisfies formally  equality in \eqref{kdv3} and we define
  \begin{equation}\label{stat3}
 v_{\bold e}(x,t)\equiv u_{\bold e}(x,t) -\phi_{\bold e}(x).
 \end{equation}
Then, for $( v_{\bold e})_{\bold e\in \bold E}\in D(A_{ext})$ we have for each $\bold e\in \bold E$
the equation 
 \begin{equation}\label{stat4}
\partial_t  v_{\bold e}= \alpha_ \bold e \partial_x ^3v_\bold e + \beta_ \bold e\partial_x v_\bold e + 2 \partial_x (v_\bold e \phi_\bold e) +  \partial_x (v^2_\bold e).
\end{equation}
Thus, we have that the system  (abusing the notation)
 \begin{equation}\label{stat5}
\partial_t  v_{\bold e}(x,t)= \alpha_ \bold e \partial_x ^3v_\bold e(x,t) + \beta_ \bold e \partial_x v_\bold e(x,t) + 2 \partial_x (v_\bold e(x,t) \phi_\bold e(x)), 
\end{equation}
represents the linearized equation for \eqref{kdv3} around $\phi_{\bold e}$. Our objective in the following will be to give sufficient conditions for obtaining that the trivial solution $v_{\bold e}\equiv 0$, $\bold e\in \bold E$, it is unstable by the linear flow of \eqref{stat5}. More exactly, we are interested in finding a {\it growing mode solution} of \eqref{stat5} with the form $ v_{\bold e}(x,t)=e^{\lambda t} \psi_{\bold e}$ and $\text{Re} (\lambda) >0$. In other words, we need to solve the formal system for $\bold e\in \bold E$,
 \begin{equation}\label{stat6}
\lambda \psi_{\bold e}=-\partial_x\mathcal L_{\bold e}  \psi_{\bold e},\qquad \mathcal L_{\bold e}=-\alpha_{\bold e}\frac{d^2}{dx^2}-\beta_{\bold e}-2\phi_{\bold e},
\end{equation}
with $\psi=( \psi_{\bold e})_{\bold e\in \bold E}\in D(\partial_x\mathcal L_{\bold e})$.

Next, we write our eigenvalue problem in \eqref{stat6} in an Hamiltonian matrix form. Indeed, for 
 $\psi=(\psi_{-}, \psi_{+})$ with $\psi_{-}=( \psi_{\bold e})_{\bold e\in \bold E_{-}}$ and  $\psi_{+}=( \psi_{\bold e})_{\bold e\in \bold E_{+}}$,  we write \eqref{stat6} as
 \begin{equation}\label{stat8} 
\lambda\left(\begin{array}{c} \psi_{-} \\ \psi_{+}\end{array}\right) =\left(\begin{array}{cc} -\partial_x \mathcal L_{-}& 0 \\0 & -\partial_x \mathcal L_{+}\end{array}\right)\left(\begin{array}{c} \psi_{-}\\ \psi_{+}\end{array}\right)\equiv NE\left(\begin{array}{c} \psi_{-}\\ \psi_{+}\end{array}\right)
 \end{equation}
 with 
 \begin{equation}\label{stat9}  
  \mathcal L_{-}=\text{diag}\Big(-\alpha_{_{1,-}}\frac{d^2}{dx^2}-\beta_{_{1-}}-2\phi_{_{1,-}},..., -\alpha_{_{m,-}}\frac{d^2}{dx^2}-\beta_{_{m, -}} -2\phi_{_{m, -}} \Big),
\end{equation} 
 where $( \alpha_{\bold e})_{\bold e\in \bold E_{-}}\equiv (\alpha_{1,-},..., \alpha_{m, -})$,  $( \beta_{\bold e})_{\bold e\in \bold E_{-}}\equiv (\beta_{1,-},..., \beta_{m,-})$, and $( \phi_{\bold e})_{\bold e\in \bold E_{-}}\equiv (\phi_{1,-},..., \phi_{m, -})$. $ \mathcal L_{+}$  being defined similarly for $( \alpha_{\bold e})_{\bold e\in \bold E_{+}}$, $( \beta_{\bold e})_{\bold e\in \bold E_{+}}$ and $( \phi_{\bold e})_{\bold e\in \bold E_{+}}$. Thus, we have that $N$ and $E$ are $(m+n)\times (m+n) $-diagonal matrix defined by
  \begin{equation}\label{stat10} 
N=\left(\begin{array}{cc} -\partial_xI_m & 0 \\ 0  & -\partial_xI_n \end{array}\right),\quad E=\left(\begin{array}{cc} \mathcal L_{-}& 0 \\0 & \mathcal L_{+}\end{array}\right),
 \end{equation} 
 where $I_k$ denotes the identity matrix of order $k$.
 
 If we denote by $\sigma(NE)=\sigma_p(NE)\cup \sigma_{ess}(NE)$ the spectrum  of $NE$ (namely, $\lambda \in \sigma_p(NE)$ if $\lambda$ is isolated and there is a $\psi\neq 0$ satisfying $ NE\psi =\lambda \psi$), the later discussion
suggests the utility of the following definition:

 \begin{definition}
The stationary vector solution $(\phi_{\bold e})_{\bold e\in \bold E}\in D(A_{ext})$    is said to be \textit{spectrally stable} for model \eqref{kdv3} if the spectrum of $NE$, $\sigma(NE)$, satisfies $\sigma(NE)\subset i\mathbb{R}.$
Otherwise, the stationary solution $(\phi_{\bold e})_{\bold e\in \bold E} $ is said to be \textit{spectrally unstable}.
\end{definition}

It is standard to show that $\sigma(NE)$ is symmetric with respect  to both the real and imaginary axes and $ \sigma_{ess}(NE)\subset i\mathbb{R}$ by supposing $N$ skew-symmetric and $E$ self-adjoint   (see, for instance, \cite[Lemma 5.6 and Theorem 5.8]{GrilSha90}). These cases on $N$ and  $E$ will be considered in our theory. Hence  it is equivalent to say that $(\phi_{\bold e})_{\bold e\in \bold E}\in D(A_{ext})$ is  \textit{spectrally stable} if $\sigma_p(NE)\subset i\mathbb{R}$, and it is spectrally unstable if $\sigma_p(NE)$ contains point $\lambda$ with  $\text{Re} (\lambda)>0.$

 It is widely known  that the spectral instability of a specific traveling wave solution of an evolution type model is   a key prerequisite to show their nonlinear instability property (see \cite{GrilSha90, Lopes, ShaStr00} and references therein). In a future work, we will study whether our spectral instability results imply nonlinear instability of stationary solutions by the KdV flow. 
 
  \subsection{Linear instability criterium} 
 
 Let $\mathcal G$ be a star graph $\mathcal G$ with a structure represented by the set $
\bold E\equiv \bold E_{-}\cup \bold E_{+}$
where $\bold E_{-}$ and $\bold E_{+}$ are finite or countable collections of semi-infinite edges $\bold e$ parametrized by $(-\infty, 0)$ or $(0, +\infty)$, respectively. The half-lines are connected at a unique vertex $\nu=0$.
 
 From \eqref{stat8}, our eigenvalue problem to solve is reduced to,
  \begin{equation}\label{stat11}    
 NE\psi =\lambda \psi, \quad Re(\lambda)>0,\;\;\psi\in D(E).
 \end{equation}  
 Next, we establish our theoretical framework and assumptions for obtaining a nontrivial solution to  problem in \eqref{stat11}:
 \begin{enumerate}
 \item[$S_1$)] Let $(A_{ext}, D(A_{ext}))$ be a extension of $(A_0, D(A_0))$ such that the solution of the  linearized KdV model \eqref{evolu} is given by a $C_0$-group.
 
  \item[$S_2$)] Suppose $0\neq \phi=(\phi_{\bold e})_{\bold e\in \bold E} \in D(A_{ext})$ such that $(\tilde{u}_{\bold e}(x,t))_{\bold e\in \bold E}=(\phi_{\bold e}(x))_{\bold e\in \bold E}$ is a  stationary solution for the KdV model \eqref{kdv3}. 
 
\item[$S_3$)] Let $E$ be   defined on a domain $D(E)\subset L^2(\mathcal G)$ on which $E$ is self-adjoint and such that $D(A_{ext})\subset D(E)$.

\item[$S_4$)]  Since for every $u\in D(A_{ext})$ we have $Eu\in D(N)$, 
we suppose $\langle NE u, \phi\rangle=0$ for every $u\in D(A_{ext})$.
  
 \item[$S_5$)] Suppose $E:D(E)\to L^2(\mathcal G)$ is  invertible  with Morse index $n(E)$ such that:
 \begin{enumerate}
\item[a)] for  $n(E)=1$, $\sigma(E)=\{\lambda_0\}\cup J_0$ with $J_0\subset [r_0, +\infty)$, for $r_0>0$, and $\lambda_0<0$,

\item[b)] for $n(E)=2$, $\sigma(E)=\{\lambda_1, \lambda_2\}\cup J$ with $J\subset [r, +\infty)$, for $r>0$, and $\lambda_1, \lambda_2<0$. Moreover, for $\Phi_1, \Phi_2, \in D(E)-\{0\}$ with $E\Phi_i=\lambda_i \Phi_i$ ($i=1,2$) we have $\langle N\phi, \Phi_1\rangle \neq 0$ or $\langle N\phi, \Phi_2\rangle \neq 0$.
 
 \end{enumerate}

  \item[$S_6$)]  For $\psi\in D(E)$ with $E\psi=\phi$, we have  $\langle \psi, \phi\rangle \neq 0 $.
 
 \item[$S_7$)] Suppose the operator  $N: D(N)\cap D(E) \to L^2(\mathcal G)$ is a skew-symmetric operator and we have that  $N$ on $D(N)$ is one-to-one.
\end{enumerate} 
 
 We note immediately from \eqref{stat8} (see Remark \ref{extension}) that the following matrix-operator relation
  $$
 NE\psi=A_{ext}\psi+ \text{diag}((2\partial_x(\phi_\bold e\psi))\delta_{ij}),\quad 1\leqq i,j\leqq |\bold E_{-}|+ |\bold E_{+}|,
 $$
implies via assumption $S_1)$, $\phi, \phi'\in L^\infty(\mathcal G)$ and   from semigroup theory (see \cite{Pa}) that the linear Hamiltonian equation
  \begin{equation}\label{hamil}
 \frac{d}{dt}{v}(t)=NE v(t)
\end{equation} 
 generates a  $C_0$-group  $\{S(t)\}_{t\in \mathbb R}$ on $L^2(\mathcal G)$. 
 
 Some of the former assumptions deserve specific comments which will be very useful  in the development of our linear instability theory.
 \begin{remark}
\item[1)]  In contrast to the classical stability theories  for solitary waves solutions on all line,  in the case of a star graph we have in general that $N\phi \notin D(E)$ (see Lemma \ref{ker} below). But from \eqref{stat2} we will have always that (see \eqref{stat9})
 $$
 \mathcal L_{+} \phi'_{+}(x)=0,\;\;\text{for}\;\;x>0,\quad \mathcal L_{-} \phi'_{-}(x)=0\;\;\text{for}\;\;x<0,
 $$
where we are writing $(\phi_{\bold e})_{\bold e\in \bold E}=(\phi_{-}, \phi_{+})$, with $\phi_{-}=(\phi_{\bold e})_{\bold e\in \bold E_{-}}$ and  $\phi_{+}=(\phi_{\bold e})_{\bold e\in \bold E_{+}}$.
 
\item[2)] From Proposition \ref{L} (the case of two half-lines) and $\phi_{\pm}$   being  either the tail or the bump profiles in \eqref{soli3}, we have for
 $\phi=(\phi_{-}, \phi_{+})$ that assumption $S_4)$, $\langle NE u, \phi\rangle=0$ for every $u\in D(A_Z)$, it  is true. Indeed, for $u=(u_{-}, u_{+})\in D(A_Z)$ defined in \eqref{domain8} follows from integration by parts  (without loss of generality we consider $\alpha_{-}=\alpha_{+}=1$ and $\beta_{-}=\beta_{+}=-1$ in \eqref{stat1})
 \begin{equation}\label{S_4}
\begin{array}{ll}
 &\int_{-\infty}^0\partial_x ( \partial^2_xu_{-})\phi_{-} dx +  \int_{0}^{+\infty}\partial_x ( \partial^2_xu_{+})\phi_{+} dx\\
 &= - \int_{-\infty}^0 u_{-} \phi'''_{-} dx- \int_{0}^{\infty} u_{+} \phi'''_{+} dx\\
 &+[ u''_{-}(0-)-u''_{+}(0+)] \phi_{+}(0+) + u'_{+}(0+) \phi'_{+}(0+)- u'_{-}(0-) \phi'_{-}(0-)\\
&=- \int_{-\infty}^0 u_{-} \phi'''_{-} dx- \int_{0}^{\infty} u_{+} \phi'''_{+} dx+ [-\frac{Z^2}{2}u_{-}(0-)-Zu'_{-}(0-)]\phi_{+}(0+) \\
&+Zu'_{-}(0-) \phi_{+}(0+) + Zu_{-}(0-) \phi'_{+}(0+)\\
&= - \int_{-\infty}^0 u_{-} \phi'''_{-} dx- \int_{0}^{\infty} u_{+} \phi'''_{+} dx+ u_{-}(0-)[Z\phi'_{+}(0+)-\frac{Z^2}{2}\phi_{+}(0+)]\\
&=- \int_{-\infty}^0 u_{-} \phi'''_{-} dx- \int_{0}^{\infty} u_{+} \phi'''_{+} dx,
 \end{array}  
\end{equation}
 where in the las equality we use the   ``even-property'' of $(\phi_{-},\phi_{+})$, namely, $\phi'_{+}(0+)=\frac{Z}{2}\phi_{+}(0+)$.
Next, since $u_{-}(0-)= u_{+}(0+)$ and $\phi_{-}(0-)= \phi_{+}(0+)$ we obtain
\begin{equation}\label{S_41}
\begin{array}{ll}
&\int_{-\infty}^0\partial_x (u_{-}-2\phi_{-}u_{-})\phi_{-} dx + \int_{0}^{+\infty}\partial_x (u_{+}-2\phi_{+}u_{+})\phi_{+} dx\\
&=- \int_{-\infty}^0 u_{-}(1-2\phi_{-})\phi'_{-} dx-\int_{0}^{+\infty}u_{+}(1-2\phi_{+})\phi'_{+} dx.
 \end{array} 
 \end{equation}
 Thus from \eqref{stat2},  \eqref{S_4} and \eqref{S_41} we obtain for $u\in D(A_L)$
 \begin{equation}
\begin{array}{ll}
& \langle NE u, \phi\rangle=\langle-\partial_x \mathcal L_{-}u_{-}, \phi_{-} \rangle+\langle-\partial_x \mathcal L_{+}u_{+}, \phi_{+} \rangle\\
 &= \int_{-\infty}^0 u_{-}(-\phi'''_{-} +\phi'_{-} - 2\phi_{-}\phi'_{-}) dx+\int_{0}^{+\infty}u_{+}(-\phi'''_{+}+\phi'_{+}-2 \phi_{+}\phi'_{+})dx=0.
\end{array} 
 \end{equation}

 \item[3)] From Proposition \ref{L}  we see that our assumption $S_3)$ in the case of a 
$\delta$-interaction for two  half-line  is not empty. Indeed, for $E=\text{diag}(\mathcal L_{-}, \mathcal L_{+})$,  with 
$\phi_{\pm}$   being  either the tail or the bump profiles in \eqref{soli3} and  with 
$$
\begin{array}{ll}
D(E)&= \{u=(u_-,u_+)\in H^2(-\infty, 0)  \oplus H^2(0,+\infty): u_{-}(0-)=u_{+}(0+),\\
&\hskip0.4in \;\text{and}\;\; u'_{+}(0+)- u'_{-}(0-)=Zu_{-}(0-)\},
\end{array} 
$$ we have the  self-adjoint property of $E$ and $D(A_Z)\subset D(E)$. Moreover, assumption $S_7)$ is immediately satisfied in this case by continuity.
 \end{remark}

 Next, we give the preliminaries for establishing our instability criterium described in Theorem \ref{crit} below. The main idea in the following is to reduce our eigenvalue problem \eqref{stat11} to  the orthogonal subspace $[\phi]^\bot$ by assumption $S_4)$. Thus we consider the orthogonal projection $Q:L^2(\mathcal G)\to L^2(\mathcal G)$
  \begin{equation}
 Q(u)=u-\langle u, \phi\rangle \frac{\phi}{\|\phi\|^2}
  \end{equation}
 associated to the nontrivial stationary solution $\phi$, and we consider 
 $$
 X_2=Q(L^2(\mathcal G))=\{f\in L^2(\mathcal G): f\bot \phi\}=[\phi]^\bot.
 $$
  We also define the closed skew-adjoint operator $N_0:D(N_0)\subset X_2\to X_2$, $D(N_0)\equiv D(N)\cap X_2$, for $f\in D(N_0)$ by
 \begin{equation}
 N_0f\equiv QN f=Nf-\langle Nf, \phi\rangle \frac{\phi}{\|\phi\|^2}
  \end{equation} 
  and the reduced self-adjoint operator for $E$, $F:D(F) \to X_2$, $D(F)=D(E)\cap X_2$ by
  \begin{equation}
 Ff\equiv QE f=Ef-\langle Ef, \phi\rangle \frac{\phi}{\|\phi\|^2}.
  \end{equation} 
 Now, for $f\in D(NE)\cap X_2=D(A_{ext})\cap X_2$  ($Ef\in D(N)$), from assumptions $S_4)$ and $S_6)$ we get the relation
 \begin{equation}\label{NF}
\begin{array}{ll}
 N_0Ff&=NEf-\langle Ef, \phi\rangle \frac{N\phi}{\|\phi\|^2}-\langle NEf-\langle Ef, \phi\rangle \frac{N\phi}{\|\phi\|^2},\phi \rangle \frac{\phi}{\|\phi\|^2}\\
 &=NEf-\langle Ef, \phi\rangle \frac{N\phi}{\|\phi\|^2}.
  \end{array} 
 \end{equation}
 
 \begin{proposition}\label{group} $N_0F: D(N_0F)\subset X_2\to X_2$, $D(N_0F)=D(A_{ext})\cap X_2\subset D(E)\cap X_2$, it is the infinitesimal generator of a strongly continuous  $C_0$-group of operators $S_0(t)$ in the space $X_2$.
\end{proposition}

\begin{proof} We divide  the proof in two steps:

\begin{enumerate}
\item[a)] Define $C= QNQEQ: D(C)\subset  L^2 (\mathcal G)\to  L^2 (\mathcal G)$, $D(C)= D(A_{ext})$. Then, for $f\in D(A_{ext})$
 \begin{equation}
\begin{array}{ll}
Cf&=NE f- \langle f, \phi\rangle \frac{NE\phi}{\|\phi\|^2}-\langle Ef, \phi\rangle \frac{N\phi}{\|\phi\|^2} +  \langle f, \phi\rangle \frac{\langle E\phi, \phi\rangle}{\|\phi\|^2} \frac{N\phi}{\|\phi\|^2}\\
&= NE f- B f
\end{array} 
 \end{equation}
 where $B:  L^2 (\mathcal G)\to  L^2 (\mathcal G)$ defined by 
 $$
 Bf=\langle f, \phi\rangle \frac{NE\phi}{\|\phi\|^2}+\langle f, E\phi\rangle \frac{N\phi}{\|\phi\|^2} -  \langle f, \phi\rangle \frac{\langle E\phi, \phi\rangle}{\|\phi\|^2} \frac{N\phi}{\|\phi\|^2},
 $$
  it is a bounded operator. Here was used that $E$ is a self-adjoint operator on $D(E)\supseteq D(A_{ext})$. Thus, from the theory of semigroups (see \cite{Pa}) $C$  generates a  strongly continuous $C_0$-group of operators $S_1(t)$ on $ L^2 (\mathcal G)$. Since $C$ commutes with $Q$,  $S_1(t)$ also commutes with $Q$.
 
\item[b)] Define $S_0(t): X_2\to X_2$ by $S_0(t)=QS_1(t)$. Then $S_0$ is a strongly continuous $C_0$-group of linear operators on $ X_2$ and it is not difficult to see that its  infinitesimal generator is $N_0F$.
\end{enumerate}
This finishes the proposition.
\end{proof}

Next, we have the following  basic assumption for our linear instability criterium in the case $n(E)=2$ in Assumption $S_5)$.

\begin{enumerate}
\item[(H)] There is a real number $\eta $, satisfying $\eta>0$, such that  $F: D(F)\to X_2$, $D(F)=D(E)\cap X_2$, it is invertible and with  Morse index equal to one. Moreover,  all the remainder of the spectrum is  contained in $[\eta, +\infty)$.
\end{enumerate}

\begin{theorem}\label{crit}
Suppose the assumptions $S_1)-S_7)$ hold with $n(E)=2$ in Assumption $S_5)$, and the basic assumption $(H)$. Then the operator $NE$ has a real positive and a real negative eigenvalue. 
\end{theorem}

The proof of Theorem \ref{crit} is based in ideas from Lopes (\cite{Lopes}) and from
the following Krasnoelskii result on closed convex cone (see \cite{Kra}, Chapter 2, section 2.2.6).

\begin{theorem}\label{Kra}
Let $K$ be a closed convex cone of a Hilbert space $(X,\|\cdot\|)$ such that there are a continuous linear functional $\Phi$ and a constant $a>0$ such that $\Phi(u)\geqq a\|u\|$ for any $u\in K$. If $T:X\to X$ is a bounded linear operator  that leaves $K$ invariant, then $T$ has an eigenvector in $K$ associated to a nonnegative eigenvalue.
\end{theorem}

\begin{proof} ({\bf{Proof of Theorem \ref{crit}}}) Our first step is to show that the operator $N_0E: D(N_0E)\subset X_2\to X_2$ has a real positive and a real negative eigenvalue. Indeed, from assumption $(H)$ we consider   $\psi_0\in D(F)=D(E)\cap X_2$, $\|\psi_0\|=1$ and $\lambda_0<0$ such that $F\psi_0=\lambda_0 \psi_0$. We define,
 \begin{equation*}
K=\{z\in D(F): \langle Fz, z\rangle \leqq 0,\;\;\text{and}\;\; \langle z, \psi_0\rangle\geqq 0 \}
\end{equation*}
 then $K$ is a nonempty closed convex cone in $L^2(\mathcal G)$. Moreover, this cone is invariant under the group $\{S_0(t)\}$. Indeed, we will use a density argument based in the existence of a core for $A_{ext}$. Thus, from semi-group theory follows that the space
 $$
 D(A_{ext}^\infty)=\bigcap_{n\in \mathbb N} D(A_{ext}^n)
 $$
 with $D(A_{ext}^n)=\{f\in D(A_{ext}^{n-1}): A_{ext}^{n-1}f\in D(A_{ext})\}$, result to be dense in $L^2(\mathcal G)$ and it is a $\{S_0(t)\}_{t\in \mathbb R}$-invariant subspace of $D(A_{ext})$. Thus, $D(A_{ext}^\infty)$ is a core for $A_{ext}$.  Therefore
is enough to consider the case $f\in K\cap D(A_{ext}^\infty)$ and so we obtain that the reduced Hamiltonian equation
 \begin{equation}
\left\{ \begin{array}{ll}
\dot{z}=N_0F z\\
z(0)= f
\end{array} \right.
 \end{equation}
 has solution $z(t)=S_0(t)f\in D(A_{ext}^\infty)$ and therefore from the self-adjoint property of $F$ and the skew-symmetric property of $N_0$ we obtain
 $$
 \frac{d}{dt}  \langle Fz(t), z(t)\rangle=  \langle FN_0Fz(t), z(t)\rangle +  \langle Fz(t), N_0Fz(t)\rangle =0,
 $$
 then for all $t$, $ \langle Fz(t), z(t)\rangle= \langle Ff, f\rangle \leqq 0$. Next, we suppose $\langle f, \psi_0\rangle> 0$ and that there is $t_0$ such that $\langle S_0(t_0) f, \psi_0\rangle<0$. Then by continuity of the flow $t\to S_0(t)f$ there is $\tau\in (0, t_0)$ with $\langle S_0(\tau)f, \psi_0\rangle=0$. Now, from  assumption $(H)$
  we have from the spectral theorem for self-adjoint operators the orthogonal decomposition for $f_\tau=S_0(\tau)f,$
  $$
 f_\tau =\sum_{i=1}^m a_i h_i +g,\quad g\bot h_i, \;\; \text{for all}\;\; i,
 $$
 where $Fh_i=\lambda_i h_i$, $\|h_i\|=1$, $\lambda_i\in \sigma_{d}(F)$ with $\lambda_i\geqq \eta$, and $\langle Fg,g\rangle\geqq \theta \|g\|^2$, $\theta>0$.
 Therefore,
 $$
 0\geqq  \langle Ff_\tau, f_\tau\rangle\geqq \sum_{i=1}^m a^2_i\lambda_i +\theta \|g\|^2\geqq \eta \sum_{i=1}^m a^2_i+\theta \|g\|^2\geqq 0.
 $$
 Thus, it follows $g=0$ and $a_i=0$ for $i$. Therefore, $S_0(\tau)f=0$ and since $S_0(t)$ is a group we obtain $f=0$ and so $\langle f, \psi_0\rangle= 0$ which is a contradiction. Now we suppose $\langle f, \psi_0\rangle= 0$, then the former analysis shows  $f=0$ and so $S_0(t)f\equiv 0$ for all $t$. It shows the invariance of $K$ by $S_0(t)$. Then, for $\mu$ large we obtain from semigroup's theory the integral representation of the resolvent  
 \begin{equation}\label{lapla}
 Tz=(\mu I-N_0F)^{-1}(z)=\int_0^\infty e^{-\mu t} S_0(t)z dt
  \end{equation}
 and it also leaves $K$ invariant. Next, for  $\Phi: L^2(\mathcal G)\to \mathbb R$ defined by $\Phi(z)= \langle z, \psi_0 \rangle$ we will see that there is $a>0$ such that $\Phi(z)\geqq a\|z\|$ for any $z\in K$. Indeed, suppose for $\|g\|=1$, $\langle g, \psi_0 \rangle=\gamma>0$  and  $\langle Fg, g\rangle\leqq 0$. Since $Ker(\Phi)$ is a hyperplane we obtain $g=z+\gamma \psi_0$ with $\langle z, \psi_0 \rangle=0$.  So, $-\lambda\gamma^2\geqq  \langle Fz, z\rangle$. Now, from the orthogonal decomposition 
  $
 z =\sum_{i=1}^m  \langle z, h_i \rangle h_i +g,\quad g\bot h_i, \;\; \text{for all}\;\; i,
 $
follows for $\eta, \theta>0$,
$\langle Fz, z\rangle=min\{\eta, \theta\}(1-\gamma^2)$. Then,
$$
\langle g, \psi_0 \rangle=\gamma\geqq \sqrt{\frac{min\{\eta, \theta\}}{-\lambda+min\{\eta, \theta\}}}\equiv a.
$$
Therefore,  by the analysis above and Theorem \ref{Kra} there are an $\alpha\geqq 0$ and a nonzero element $\omega_0\in K$ such that $
 (\mu I -N_0F)^{-1}(\omega_0)=\alpha \omega_0$. It is immediate that $\alpha>0$ and so $N_0F\omega_0 = \zeta \omega_0$ with $\zeta=\frac{\mu \alpha-1}{\alpha}$. Next we see that $\zeta\neq 0$. Suppose that $\zeta= 0$, then from \eqref{NF} and the injectivity of $N$ we obtain
 $$
 E\omega_0=\langle E\omega_0, \phi\rangle \frac{\phi}{\|\phi\|^2}.
 $$
 From assumption $S_5)$, let $\psi\in D(E)$ with $E\psi=\phi$, then since $E$ is invertible follows
 $$
 \omega_0= \frac{\langle E\omega_0, \phi\rangle}{\|\phi\|^2}\psi\;\;\text{and}\;\; 0=\langle \omega_0, \phi\rangle= \frac{\langle E\omega_0, \phi\rangle}{\|\phi\|^2}\langle \psi, \phi\rangle.
 $$
 Since $\langle \psi, \phi\rangle\neq 0$ follows $\langle E\omega_0, \phi\rangle=0$. Hence $E\omega_0=0$ and so $\omega_0=0$, which is a contradiction. Then, $N_0F$ has a nonzero real eigenvalue $\zeta$.
 
 Now, we have $\sigma(N_0F)=\sigma((N_0F)^*)=-\sigma(FN_0)=
 -\sigma(FN_0FF^{-1})=-\sigma(N_0F)$ and so $-\zeta$ also belongs to  $\sigma(N_0F)$. Thus from Theorem 5.8  of \cite{GrilSha90}, the essential spectrum of $N_0F$ lies on the imaginary axis and then $-\zeta$ is an eigenvalue of $N_0F$ and this proves the claim.
 
 Thus, for $\omega_0 \in D(N_0F)$, $\omega_0 \neq 0$, and $ \zeta>0$ we have,
  \begin{equation}\label{final}
NE\omega_0=\langle E\omega_0, \phi\rangle  \frac{N\phi}{\|\phi\|^2}+\zeta \omega_0.
 \end{equation}
 Next we consider two cases:
 \begin{enumerate}
 \item[a)] Suppose $\langle E\omega_0, \phi\rangle=0$, then $NE\omega_0=\zeta \omega_0$ and the proof of the criterium finishes.
  \item[b)] Suppose $r\equiv \frac{1}{\|\phi\|^2}\langle E\omega_0, \phi\rangle\neq0$. From Assumption $S5)$, we consider
  $$
  u=\omega_0 + a \Phi_1 + b\Phi_2, \;\; E\Phi_i=\lambda_i \Phi_i, \;1\leqq i \leqq 2,
  $$
  with $\|\Phi_i\|=1$, $\Phi_1\bot \Phi_2$. We will find $a, b\in \mathbb{R}$, not both zero, such that $NEu= \zeta u$ and $u\neq 0$. Thus, we obtain initially the relation
  \begin{equation}\label{final1} 
 r N\phi +a \lambda_1 N\Phi_1 + b \lambda_2 N\Phi_2=a \zeta \Phi_1 + b \zeta \Phi_2.
  \end{equation}
  Therefore, from the skew-symmetric property of $N$ we obtain the system
 \begin{equation}\label{ab}
\left\{ \begin{array}{ll}
a \zeta +b \lambda_2\langle N\Phi_1, \Phi_2\rangle=r \langle N\phi, \Phi_1\rangle\\
a \lambda_1\langle N\Phi_1, \Phi_2\rangle-\zeta b=-r \langle N\phi, \Phi_2\rangle.
\end{array} \right.
 \end{equation} 
  Thus, since the determinant of the coefficients   is different of zero ($\zeta^2 +\lambda_1\lambda_2 \langle N\Phi_1, \Phi_2\rangle^2\neq 0$), $r\neq 0$ and from  Assumption $S5)$, we obtain a nontrivial solution for  \eqref{ab}. 
  
Next we see $u\neq 0$. Indeed, suppose $u=0$. Then, from relation $\omega_0=-a \Phi_1 -b\Phi_2$ and by substituting in \eqref{final} we obtain 
 \begin{equation}\label{ab2}
a \lambda_1 r \langle N\phi, \Phi_1\rangle + b \lambda_2 r \langle N\phi, \Phi_2\rangle=0.
 \end{equation} 
 Then, by using system \eqref{ab} in \eqref{ab2} we arrive to the relation $\zeta (a^2 \lambda_1+ b^2  \lambda_2)=0$, it which is a contradiction.
  It is proves Theorem \ref{crit}.
  \end{enumerate}
\end{proof}

Next, we consider the case $n(E)=1$ in Assumption $S_5)$.

\begin{theorem}\label{crit2}
Suppose the assumptions $S_1), S_2), S_3), S_5), S_7)$ hold with $n(E)=1$. Then the operator $NE$ has a real positive and a real negative eigenvalue. 
\end{theorem}

\begin{proof} In this case we do not need to reduce the eigenvalue problem \eqref{stat11} to the orthogonal subspace $[\phi]^\bot$. Indeed, from assumption $S_1)$ we have  that $NE$ is the infinitesimal generator of a $C_0$-group $\{S(t)\}_{t\in \mathbb R}$.  For $\psi_0\in D(E)$, $\|\psi_0\|=1$ and $\lambda_0<0$ such that $E\psi_0=\lambda_0 \psi_0$, we consider the following nonempty closed convex cone 
 \begin{equation*}
K_0=\{z\in D(E): \langle Ez, z\rangle \leqq 0,\;\;\text{and}\;\; \langle z, \psi_0\rangle\geqq 0 \}.
\end{equation*}
Similarly as in the proof of Theorem \ref{crit}, $K_0$ is invariant by the group $S(t)$. Thus, for $T=(\mu I- NE)^{-1}$, $\mu$ large, $T$ leaves $K_0$ invariant. Then, by using Theorem \ref{Kra} with this $T$ and $\Phi(z)=\langle z, \psi_0\rangle$, we can see that  $NE$ has a real positive and a real negative eigenvalue. This finishes the proof.
\end{proof}
 
 \subsection{One application of Theorem \ref{crit}}

 The following framework will be used in the study of linear instability of bump and tail profiles on star graph. Suppose that  assumptions $S_1)-S_7)$ above hold with $n(E)=2$ and for $\psi$ such that  $E\psi=\phi$ we have $\langle \psi, \phi\rangle< 0$. Then assumption $(H)$ is true. Indeed, from
 assumption $S_6)$ we obtain that $F$ is invertible.  Next,  there are $a, b\in \mathbb R$ (not both zeros) with $\langle a\Phi_1+b \Phi_2, \phi\rangle=0$ and $
\langle F(a\Phi_1+b \Phi_2), a\Phi_1+b \Phi_2 \rangle<0$. Then via min-max principle we have $n(F)\geqq 1$. Next, suppose that $n(F)=2$ and consider $z_1, z_2\in X_2$, $z_1\bot z_2$, $\mu_1, \mu_2<0$, and $Fz_i=\mu_i z_i$. Then we get
$$
\langle Fz_i, z_i \rangle =\langle Ez_i, z_i \rangle =\mu_i\|z_i\|^2<0, \;\;\text{and}\;\; \langle Ez_1, z_2 \rangle=0.
$$
 Moreover, since $\psi\notin X_2=[\phi]^\bot$ follows that set $\{\psi, z_1, z_2\}\subset E$ is linearly independent and we have the relations
 $$
\langle Ez_i, \psi  \rangle = \langle z_i,  \phi \rangle=0, \;\;\text{and},\;\; \langle E\psi, \psi \rangle=\langle \phi, \psi \rangle<0.
$$
 Therefore $\langle E(\alpha \psi+\beta z_1+ \theta z_2), \alpha \psi+\beta z_1+ \theta z_2\rangle<0$ and so $n(E)\geqq 3$, it which is not true. Then $n(F)=1$ and all other eigenvalues (and the remain of the spectrum) are contained in $[\eta, +\infty)$. Thus, from Theorem \ref{crit} follows that $NE$ has a real positive and a real negative eigenvalue.

 \section{Linear instability of tail and  bump solutions for the KdV on two half-lines}
 
 The focus of this section is to apply the linear instability criterium in  Theorems \ref{crit} and \ref{crit2} to the KdV on a star graph
 with two half-lines and a $\delta$-interaction-type at the vertex $\nu=0$. Our main result is the following,
 
 \begin{theorem}\label{main} For $Z\neq 0$, $\alpha_{-}=\alpha_{+}>0$, $\beta_{-}=\beta_{+}<0$, $-\frac{\beta_{+}}{\alpha_{+}}>\frac{Z^2}{4}$,  let  $\phi_Z\equiv (\phi_{-}, \phi_{+})  \in D(A_Z)$ defined for $\phi_{+}(x)$  by the formula \eqref{soli6} with  $x>0$ and $\phi_{-}(x)=\phi_{+}(-x)$ for $x<0$.  We consider the following family of  stationary solutions for the Korteweg-de Vries model \eqref{kdv3} on the star graph $\mathcal G$ with $\bold E=(-\infty, 0) \cup (0, +\infty)$,
$$
U(x,t)=(\phi_{-}(x), \phi_{+}(x)),\quad t\in \mathbb R.
$$ 
Then,   this  family  of tail ($Z<0$) and bump ($Z>0$) profiles  are linearly unstable.  
   \end{theorem} 

The proof of Theorem \ref{main} will be divided in several steps and we consider the cases  $\alpha_{-}=\alpha_{+}=1$, $\beta_{-}=\beta_{+}=-1$ and $1>\frac{Z^2}{4}$, without loss of generality.  From Proposition \ref{L},  assumption $S_1)$ is filled by   $(A_Z, D(A_Z))$ defined in \eqref{domain8}. The linear eigenvalue problem to be solve \eqref{stat11} for $\lambda>0$, it is determined by the matrices  $N, E$ in \eqref{stat10} with the Schr\"odinger operators on half-lines
$$
\mathcal L_{\pm}=- \frac{d^2}{dx^2}+1 -2\phi_{\pm}.
$$
 The domain for $E=E_Z$ is given in $H^2(\mathcal G)=H^2(-\infty, 0)  \oplus H^2(0,+\infty)$ for  $Z\in \mathbb R$ by 
 \begin{equation}\label{L_i}
D(E_Z)= \{(u_{-},  u_{+})\in H^2(\mathcal G) : u_{-}(0-)=u_{+}(0+), u'_{+}(0+)- u'_{-}(0-)=Zu_{-}(0-)\},
\end{equation}
and so $(E_Z, D(E_Z))$ represents a self-adjoint family of point interactions on all the line by identifying $D(E_Z)$ as $H^2(\mathbb R-\{0\})\cap H^1(\mathbb R)$. Moreover, it is immediate that  $D(A_Z)\subset D(E_Z)$ (assumption $S_3)$). From Remark 4.2-item $2)$ we have assumption $S_4)$. Assumption $S_7)$ follows by continuity.

 The following lemma implies that $E_Z$ is invertible (assumption $S_5)$).

\begin{lemma}\label{ker} For every $Z\neq 0$ we have $Ker(E_Z)=\{0\}$. Moreover, since $\sigma_{ess}(E_Z)=[1, +\infty)$ we obtain $E_Z: D(E_Z)\to L^2(\mathcal G)$ is invertible.
\end{lemma}
 
 \begin{proof} Let $u=(u_{-},  u_{+})\in D(E_Z)$, $E_Zu=0$. Since $\mathcal L_{\pm} \phi'_{\pm}=0$, we need to have $u_{-}(x)=a \phi'_{-}(x)$, $x<0$, and $u_{+}(x)=b\phi'_{+}(x)$, $x>0$ (see \cite{BerShu91}). Next, from the continuity property at zero for $u$, $ \phi'_{+}(0+)=-\phi'_{-}(0-)$, and $ \phi''_{+}(0+)=\phi''_{-}(0-)$ we have that
 \begin{equation}\label{1L}
 a=-b,\;\;\text{and}\;\; -2a \phi''_{+}(0+)=Zu_{+}(0+)=Zu_{-}(0-)=Za\phi'_{-}(0-)=-\frac{Z}{2}a\phi_{+}(0+).
  \end{equation}
Suppose $a\neq 0$. Then, from  \eqref{1L} we have $\phi''_{+}(0+)=\frac{Z^2}{4}\phi_{+}(0+)$ and so from \eqref{soli1} we arrive to
$$
1-\phi_{+}(0+)=\frac{Z^2}{4} \Longrightarrow Z^2=4,
$$ 
 it which does not happen ($1>\frac{Z^2}{4}$). Then, $a=b=0$ and $u\equiv 0$.
 
 Next, by Weyl's theorem (see Theorem XIII.14 of \cite{RS}), the essential spectrum of $E_Z$ coincides with $[1, +\infty)$. Then $E_Z$ is an invertible operator. This finishes the proof.
\end{proof}

 \begin{lemma}\label{Morse} For  $Z> 0$ we have $n(E_Z)=2$ and for  $Z< 0$ that $n(E_Z)=1$.
\end{lemma}
 
  \begin{proof}  Our  strategy  is to use analytic  perturbation theory (see  \cite{AngGol17a, AngGol17b}). For this purpose we define the self-adjoint operator on $L^2(\mathbb R)$
\begin{equation}\label{limi}  
 \mathcal L_0=-\frac{d^2}{dx^2}+1-2 \phi_0,\quad D(\mathcal L_0)=H^2(\mathbb R)
 \end{equation}
 where $\phi_0$ denotes the classical one soliton solution for the KdV equation on the full line,
 \begin{equation}\label{soli}
\phi_{0}(x)=\frac{3}{2} sech^2\Big(\frac{1}{2} x\Big),\;\;x\in \mathbb R.
\end{equation}
 From classical Sturm-Liouville theory $Ker(\mathcal L_0)=[\phi'_{0}]$, $n(\mathcal L_0)=1$, $\sigma_{ess}(\mathcal L_0)=[1, +\infty)$ (see \cite{BerShu91}). Now, we consider the domain 
  \begin{equation}
 D(E_0)=\{(u_{-},  u_{+})\in H^2(\mathcal G) : u_{-}(0-)=u_{+}(0+), u'_{-}(0-)=u'_{+}(0+)\}.
 \end{equation}
 on which the following ``limit'' operator $E_0$ (associated with $E_Z$) is self-adjoint
  \begin{equation}
 E_0=\left(\begin{array}{cc} -\frac{d^2}{dx^2}+1-2\phi_{0, -}& 0 \\0 &-\frac{d^2}{dx^2}+1-2\phi_{0,+} \end{array}\right),
 \end{equation}
 with $\phi_{0, -}=\phi_{0}|_{(-\infty, 0)}$ and $\phi_{0, +}=\phi_{0}|_{(0,+\infty)}$. Thus, by considering the following unitary operator $\mathcal U: D(E_0) \to H^2(\mathbb R)$ defined for $u=(u_{-},  u_{+})\in E_0 $ by $\mathcal U(u)= \tilde{u}\in H^2(\mathbb R)$ where
 \begin{equation}
\tilde{u}=\left\{ \begin{array}{lll}
u_{-}(x),&\quad x<0\\
u_{+}(x),&\quad x>0\\
u_{+}(0+),&\quad x=0,
\end{array} \right.
 \end{equation} 
we obtain   $\sigma(E_0)=\sigma(\mathcal L_0)$ and  $\lambda\in \sigma_{disc}(E_0)$ if and only if  $\lambda\in \sigma_{disc}(\mathcal L_0)$ with the same multiplicity. Moreover, $\sigma_{ess}(E_0)=[1, +\infty)$. Therefore, $Ker(E_0)=[\Phi'_0]$, $\Phi_0=(\phi_{0, -}, \phi_{0, +})$,  and $n(E_0)=1$.
 
 Next, by using a similar strategy as in  \cite{AngGol17a, AngGol17b}  for studying the stability of  standing wave solutions for nonlinear Schr\"odinger models on star graphs, we have the following:
   \begin{enumerate}
 \item[i)]  $\phi_Z=(\phi_{-}, \phi_{+})\to \Phi_0$, as $Z\to 0$, in $H^1(\mathcal G)$.
 
 \item[ii)] The   family  $\{E_Z\}_{Z\in \mathbb R}$ represents a  real-analytic family of self-adjoint operators of type (B) in the sense of Kato (see \cite{kato}). 
 
  \item[iii)]  Since $E_Z$ converges to $E_0$ as $Z \to 0$ in the generalized sense, we obtain from  Theorem IV-3.16 from Kato \cite{kato} and from Kato-Rellich Theorem (\cite{RS}, Theorem XII.8) the existence of two analytic functions $\Omega, \Pi$ defined in a neighborhood of zero with $\Omega: (-Z_0, Z_0)\to \mathbb R$ and $\Pi: (-Z_0, Z_0)\to L^2(\mathcal{G})$ such that
$\Omega(0)=0$ and $\Pi(0)=\Phi'_0$. For all $Z\in (-Z_0, Z_0)$, $\Omega(Z)$ is the  simple isolated second eigenvalue of $E_Z$, and $\Pi(Z)$ is the associated eigenvector for $\Omega(Z)$. Moreover, $Z_0$ can be chosen small enough to ensure that for  $Z\in (-Z_0,Z_0)$  the spectrum of $E_Z$ in $L^2(\mathcal{G})$ is positive, except at most the  first two eigenvalues.

\item[iv)] From a ODE's analysis we have that if $\lambda$ is an simple eigenvalue for $E_Z$ then the eigenfunction associated is either even or odd. Therefore, since $\Pi(Z)\to \Phi'_0$, as  $Z\to 0$, and $ N\Phi_0$ is odd, we can see that $\Pi(Z)\in H^2(\mathbb R)$ is a odd function. Thus we obtain the relation
\begin{equation}\label{S5}
\langle N\phi_{Z}, \Pi(Z)\rangle \neq 0,\quad Z\approx 0.
\end{equation}
Indeed, since  $\lim_{Z\to 0}\langle N\phi_{Z}, \Pi(Z)\rangle=\|N\Phi_{0}\|^2>0$, we have  for $Z$ small property \eqref{S5}.  Thus, an continuation argument shows \eqref{S5} for all $Z \in (-\infty, \infty)$.

 \item[v)]  From Taylor's theorem we see that there exists $0<Z_1<Z_0$ such that $\Omega(Z)>0$ for any $Z\in (-Z_1,0)$, and $\Omega(Z)<0$ for  any $Z\in (0, Z_1)$.  Thus, in the space  $ L^2(\mathcal G)$ for $Z$ small,  we have  $n(E_Z)=1$ as $Z<0$, and $n(E_Z)=2$ as $Z>0$. 

  \item[vi)]  Recall that  $Ker(E_Z)=\{0\}$ for $Z\neq 0$. Thus, we define $Z_\infty$ by
$$
Z_\infty=\sup \{\tilde{Z}>0:  E_Z\;{\text{has exactly two negative eigenvalues for all}}\; Z\in (0,\tilde{Z})\}.
$$
Item $v)$ above implies that $Z_\infty$ is well defined and $Z_\infty\in (0,\infty]$. We claim that $Z_\infty=\infty$. Suppose that $Z_\infty< \infty$. Let $M=n(E_{Z_\infty})$ and $\Gamma$  be a closed curve (for example, a circle or a rectangle) such that $0\in \Gamma\subset \rho(E_{Z_\infty})$, and  all the negative eigenvalues of  $E_{Z_\infty}$ belong to the inner domain of $\Gamma$.  The existence of such $\Gamma$ can be  deduced from the lower semi-boundedness of the quadratic form associated to $E_{Z_{\infty}}$.
 
Next, from item $ii)$ above  follows  that there is  $\epsilon>0$ such that for $Z\in [Z_{\infty}-\epsilon, Z_{\infty}+\epsilon]$ we have $\Gamma\subset \rho(E_Z)$ and for $\xi \in \Gamma$,
$Z\to (E_Z-\xi I_d)^{-1}$ is analytic (see \cite{RS4}). Therefore, the existence of an analytic family of Riesz-projections $Z\to P(Z)$  given by 
$$
P(Z)=-\frac{1}{2\pi i}\oint_{\Gamma} (E_{Z}-\xi I_d)^{-1}d\xi
$$
implies  that $\dim(range (P(Z)))=\dim(range (P(Z_\infty)))=M$ for all $Z\in [Z_\infty-\epsilon, Z_{\infty}+\epsilon]$. Next, by definition of $Z_\infty$, $E_{Z_{\infty}-\epsilon} $ has two negative eigenvalues, and $M=2$, hence $E_Z$ has two negative eigenvalues for $Z\in (0,Z_{\infty}+\epsilon]$,  which contradicts with the definition of $Z_{\infty}$. Therefore,  $Z_{\infty}=\infty$.  

Analogously we can prove that $n(E_Z)=1$ in the case $Z<0$. This finishes the proof.
\end{enumerate}
 \end{proof}
 


The following lemma shows assumption $S_6)$ in the case $n(E)=2$. Indeed, by returning to the variable $\beta_{+}$ defining  the profiles $\phi_{\pm}$ in \eqref{soli6} with $\alpha_{+}=1$, we have that these profiles represent a differentiable family of stationary solutions a one-parameter $\omega=-\beta_{+}>0$ and we can denote it  dependence as $\phi_{Z,\omega}=(\phi_{-, \omega}, \phi_{+, \omega})$. From \eqref{soli1} we obtain after derivation in $\omega$ that
\begin{equation}\label{derivada}
\Big (-\frac{d^2}{dx^2}+\omega-2 \phi_{\pm, \omega}\Big) \Big(\frac{d}{d\omega} \phi_{\pm, \omega} \Big)=- \phi_{\pm, \omega}.
\end{equation}
Next, by denoting $\psi_\omega= (- \frac{d}{d\omega} \phi_{-, \omega}, -\frac{d}{d\omega} \phi_{+, \omega})$ is not difficult to see that $\psi\equiv \psi_\omega|_{\omega=1}\in D(E_Z)$ and so the expression $E_Z \psi=\phi_{Z}$ makes sense. Thus, with the notation above, we obtain the following result.

 \begin{lemma}\label{S_5}
   Let $Z\neq0$. The smooth curve of profiles $\omega \in (\frac{Z^2}{4},+\infty)\to \phi_{Z,\omega}$ with formula \eqref{soli6} satisfies for  $\psi\equiv -\frac{d}{d\omega} \phi_{Z,\omega}|_{\omega=1}$ the relations: $\psi \in D(E_Z)$,
 \begin{equation}\label{condi}
 E_Z \psi= \phi_Z,\quad\text{and},\quad \langle \psi, \phi_{Z} \rangle<0.
 \end{equation} 
\end{lemma}

\begin{proof}  
 From Proposition 3.19 in \cite{AngGol17b}  (item (ii), $p=2$) we have   for every $Z\in \mathbb R$, the relation 
 $$
  \frac{d}{d\omega}\Big[\int_{-\infty}^0 \phi^2_{-, \omega}(x)dx + \int_{0}^{+\infty} \phi^2_{+, \omega}(x)dx\Big]>0.
 $$
 Therefore $\langle \psi, \phi_{Z} \rangle<0$. This finishes the proof.
\end{proof}  

\begin{proof}[{\bf{Proof of Theorem \ref{main}}}]
Let $Z>0$.  From  Lemmas \ref{ker}, \ref{Morse} and \ref{S_5}, subsection 4.2 and \eqref{S5}, follows from Theorem \ref{crit} that   the profiles of type  bump for the KdV  are linear unstable. 

Let $Z<0$. From  Lemmas \ref{ker} and \ref{Morse}, Theorem \ref{crit2} implies the  linear instability of the tail profile for the KdV model. This finishes the proof.
\end{proof}

\section{Linear instability of the tail and  bump solutions for a balanced general star graph $\mathcal G$}

 The focus of this section will be consider the KdV model \eqref{kdv3} on a balanced metric star graph $\mathcal G$ with a structure $\bold E\equiv \bold E_{-}\cup \bold E_{+}$ where $|\bold E_{+}|=|\bold E_{-}|=n$, $n\geqq 2$, and with a $\delta$-interaction at the vertex. Thus by following the notation in \cite{MNS} and  Section 2 above, for $I=I_{n\times n}$ being the identity matrix of order  $n\times n$ we consider the matrix $L\equiv L_{3n\times 3n}:\mathbb G_{-}\to \mathbb G_{+}$ of order $3n\times 3n$, for $Z\in \mathbb R$ as
 \begin{equation}\label{L2}
 L\equiv  \left(\begin{array}{ccc} I & 0 & 0 \\Z I & I & 0 \\\frac{Z^2}{2} I & ZI & I \end{array}\right).
\end{equation}
 Thus, from \eqref{domain4} and with  $\alpha_{\pm} =(\alpha_ \bold e)_{\bold e\in \bold E_{\pm}}$, $\beta_{\pm} =(\beta_ \bold e)_{\bold e\in \bold E_{\pm}}$, we obtain  $L^tB_{+}L=B_{-}$ if and only if $\alpha_{+}=\alpha_{-}$ and $\beta_{+}=\beta_{-}$. Then, in this case (and only in this one) we obtain that $L$ is $(\mathbb G_{-}, \mathbb G_{+})$-unitary. Therefore, the operators $(H_Z, D(H_Z))$ defined by  
\begin{equation}\label{L_Z}
\left\{ \begin{array}{lll}
H_Zu=-A_0^*u=A_0 u\\
D(H_Z)=\Big\{u\in  D(A_0^*): L(u(0-), u'(0-), u''(0-))=(u(0+), u'(0+), u''(0+))\Big \}
\end{array} \right.
\end{equation}
are a skew-self-adjoint family of extension for $(A_0, D(A_0))$, where for  $u=(u_ \bold e)_{\bold e\in \bold E}\in D(H_Z)$ we have used the abbreviations 
$$
u(0-)=(u_ \bold e(0-))_{\bold e\in \bold E_{-}},\;\; u'(0-)=(u'_ \bold e(0-))_{\bold e\in \bold E_{-}},\;\; u''(0-)=(u''_ \bold e(0-))_{\bold e\in \bold E_{-}},
$$
 (similarly for the terms $u(0+)$, $u'(0+)$ and $u''(0\pm)$). Thus, we obtain the following system of conditions
\begin{equation}\label{condiZ}
\begin{aligned}
&u(0-)=u(0+), \quad u'(0+)- u'(0-)=Zu(0-),\;\;(\delta-\text{interaction type})\\
& \quad \quad \frac{Z^2}{2}u(0-)+Zu'(0-)=u''(0+)-u''(0-).
\end{aligned}
\end{equation}

Now, we build a family of continuous (at zero) stationary profiles for the KdV model on the balanced graph $\mathcal G$. By abusing of the notation, it consider the constants sequences $(\alpha_ \bold e)_{\bold e\in \bold E}=(\alpha_+)$, $(\beta_ \bold e)_{\bold e\in \bold E}=(\beta_+)$, with $\alpha_+>0$ and $\beta_+<0$. Thus we obtain a system of $n$-KdV models equals defined on all the line. Then, for $Z>0$ and  $-\frac{\beta_+}{\alpha_+}>\frac{Z^2}{4}$ we consider the half-soliton profile $\phi_{+}$ defined in \eqref{soli6}  and  $\phi_{-}(x)\equiv \phi_{+}(-x)$ for $x<0$. We define the constants sequences of functions  $u_{-}=(\phi_{-})_{\bold e\in \bold E_{-}}$,  $u_{+}=(\phi_{+})_{\bold e\in \bold E_{+}}$, and so  $U_{Z}\equiv U_{Z, \alpha_+, \beta_+}=(u_{-}, u_{+})$ represents a family of stationary bump profiles  for the KdV model in \eqref{kdv3} (see Figure 4) and satisfying the boundary conditions \eqref{condiZ}.
 \begin{figure}[h]\label{bumpgraph}
	\centering
	\includegraphics[angle=0,scale=1]{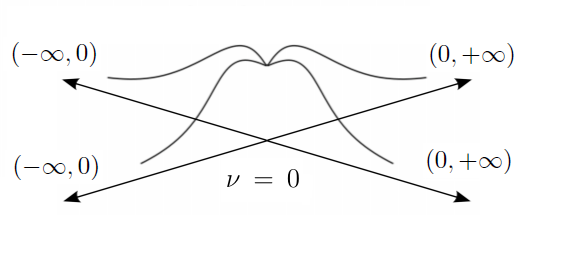}
	\caption{Bump profiles in a balanced star graphs with four edges}
\end{figure}
The case $Z<0$, $U_{Z}$ represents the corresponding family of stationary tail profiles (see Figure 5).

With the notations above, the main result of this section is the following.

\begin{theorem}\label{main2} Let $Z\neq 0$.  For $\alpha_{+}>0$ and  $0>\beta_{+}$, $-\frac{\beta_{+}}{\alpha_+}>\frac{Z^2}{4}$, we consider the profiles $\phi_{\pm}$  in \eqref{soli6}.  Define $U_{Z}= (\phi_{\bold e})_{\bold e\in \bold E}\in D(H_Z)$ with $\phi_{\bold e}= \phi_{-}$ for $\bold e\in \bold{E}_{-}$ and   $\phi_{\bold e}= \phi_{+}$ for $\bold e\in \bold{E}_{+}$. Then, 
 $$
\Phi_{Z}(x,t)=U_{Z}(x)
$$ 
defines a family  of linearly unstable stationary solutions 
 for the  Korteweg-de Vries model \eqref{kdv3}. 
   \end{theorem} 
   
  For the general case of the sequences $(\alpha_ \bold e)_{\bold e\in \bold E}$ and  $(\beta_ \bold e)_{\bold e\in \bold E}$, in Remark \ref{geral} below we establish  the necessary conditions for obtaining the linear instability of the corresponding bump and tail profiles.

The linear instability of the {\it continuous (at zero) tail and bump profiles} $U_{Z}$, $Z\neq 0$, it will be a consequence of 
Theorems \ref{crit}-\ref{crit2} with a framework determined by the   space $D(H_Z)\cap \mathcal C$ where
  \begin{equation}\label{conti} 
 \mathcal C =\{(u_ \bold e)_{\bold e\in \bold E}\in L^2(\mathcal G): u_{1, -}(0-)=...=u_{n, -}(0-)=u_{1, +}(0+)=...=u_{n, +}(0+)\},
  \end{equation}
 represents the set of elements of $L^2(\mathcal G)$ continuous at the graph-vertex $\nu=0$. Thus, by following the notation in Section 5 ($(\alpha_{\bold e})_{\bold e\in \bold E}=(1)_{\bold e\in \bold E}$,  $(\beta_{\bold e})_{\bold e\in \bold E}=(-1)_{\bold e\in \bold E}$, without loss of generality) we start our analysis  by considering the $2n\times 2n$-matrix derivate operator $N$ in \eqref{stat10} and the $2n\times 2n$-matrix Schr\"odinger operator 
  \begin{equation}\label{tail1}
 \mathcal E_Z=\left(\begin{array}{cc} \mathcal L_{Z, -}& 0 \\0 &  \mathcal L_{Z, +} \end{array}\right).
 \end{equation}
 with  
 \begin{equation}\label{tail2}  
  \mathcal L_{Z,\pm}=\text{diag}\Big(-\frac{d^2}{dx^2}+1 -2\phi_{\pm},..., -\frac{d^2}{dx^2}+ 1 -2\phi_{\pm} \Big),
\end{equation} 
 being $n\times n$-diagonal matrices.
 
 From the proof of Lemma \ref{Morse12} below,  $\mathcal E_Z$ is a family  of self-adjoint operators with domain $D(\mathcal E_Z)=D_{Z,\delta}\cap \mathcal C\subset H^2(\mathcal{G})$ with
  \begin{equation}\label{tail3} 
u\in D_{Z,\delta} \Leftrightarrow  u(0-)=u(0+), \;\;\sum_{\bold e\in \bold E_{+}}u'_{\bold e}(0+)- \sum_{\bold e\in \bold E_{-}}u'_{\bold e}(0-)=Znu_{1, +}(0+).
 \end{equation} 
 
 It is immediate from \eqref{condiZ} that $D(H_Z)\cap \mathcal C\subset D(\mathcal E_Z)$ and so assumption $S_3)$ holds. From Remark 4.2-item 2) we obtain again assumption $S_4)$. Assumption $S_7)$ is immediate by the continuity property at zero of each element in $D(\mathcal E_Z)$. Moreover, from Proposition \ref{group28} (Appendix) we have that subspace $D(H_Z)\cap \mathcal C$ is invariant  by the unitary group $\{W(t)\}_{t\in \mathbb R}$ generated by $H_Z$.
 
 The proof of the following result follows the same strategy as in Lemma \ref{ker}.
 
  \begin{lemma}\label{Morse10} Let $Z\neq  0$ and the operator $\mathcal E_Z:D(\mathcal E_Z)\to L^2(\mathcal{G})$ defined in \eqref{tail1} with $D(\mathcal E_Z)=D_{Z,\delta}\cap \mathcal C$. Then,
$ \mathcal E_Z$ is invertible with $\sigma_{ess}(\mathcal E_Z)=[1,+\infty)$.
 \end{lemma}

 \begin{proposition}\label{Morse11} Let $\mathcal E_Z:D(\mathcal E_Z)\to L^2(\mathcal{G})$ defined in \eqref{tail1} with $D(\mathcal E_Z)=D_{Z,\delta}\cap \mathcal C$. Define the following closed subspace   on $L^2(\mathcal G)$,
 $$
 L^2_{n}(\mathcal G)=\{u=(u_{\bold e})_{\bold e\in \bold E}: u_{\bold e}=f, \;\;\text{for all}\; \bold e\in \bold E_{-}, u_{\bold e}=g, \;\;\text{for all}\; \bold e\in \bold E_{+}\} 
$$
 Then, $n(\mathcal E_Z|_{L^2_{n}(\mathcal G)})=2$, for $Z>0$, and $n(\mathcal E_Z|_{L^2_{n}(\mathcal G)})=1$, for $Z<0$.
 \end{proposition}
 
 The proof of  Proposition \ref{Morse11} will based in the analytic perturbation theory and the extension theory of symmetric operators. We note that in the case $Z<0$ (tail case) can be given an argument based exclusively in the extension theory of symmetric operators and to be obtained that $n(\mathcal E_Z)=1$ on $L^2(\mathcal G)$.
  \begin{figure}
 	\centering
 	\includegraphics[angle=0,scale=1]{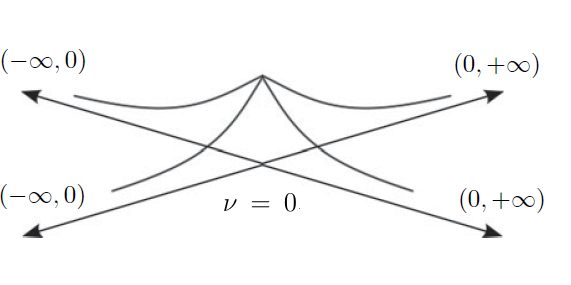}
 	\caption{Tail profiles in a balanced star graphs with four edges.}
 	\label{fig:
 		tail11}
 \end{figure}

 The proof of Proposition \ref{Morse11} will be divide in several lemmas.
 
  \begin{lemma}\label{Morse12} Define the self-adjoint matrix Schr\"odinger operator in $L^2(\mathcal G)$ with Kirchhoff's type condition at $\nu=0$
   \begin{equation}\label{Morse 13}
 \mathcal E_0=\left(\begin{array}{cc} \mathcal L_{0, -}& 0 \\0 &  \mathcal L_{0, +} \end{array}\right)
 \end{equation}
 where 
 \begin{equation}\label{tail2}  
  \mathcal L_{0,\pm}=\text{diag}\Big(-\frac{d^2}{dx^2}+ 1 -2\phi_{0},..., -\frac{d^2}{dx^2}+ 1 -2\phi_{0} \Big),
\end{equation} 
 being $n\times n$-diagonal matrices, $\phi_0$ the soliton profile defined in \eqref{soli}, and 
\begin{equation}\label{E_0}
D(\mathcal E_0 )= \{u\in H^2(\mathcal G) \cap \mathcal C: u(0-)=u(0+),  \sum_{\bold e \in \bold E_{+}} u'_{\bold e}(0+)-\sum_{\bold e \in \bold E_{-}} u'_{\bold e}(0-)=0\}.
\end{equation} 
  \begin{enumerate}
 \item[1)] In the space $L^2_{n}(\mathcal G)$ we have $Ker(\mathcal E_0)=[\Phi'_0]$, where $\Phi'_0=(\phi_0')_{\bold e\in \bold E}$.
 \item[2)] The operator $(\mathcal E_0, D(\mathcal E_0))$ has one simple negative eigenvalue in $L^2(\mathcal G)$. Moreover, we have also $n(\mathcal E_0|_{L^2_{n}(\mathcal G)})=1$.
  \item[3)]  The rest of the spectrum of $\mathcal E_0 $ is positive and bounded away from zero.
  \end{enumerate}
    \end{lemma}

\begin{proof} The proof of item $1)$  follows from a similar analysis as in Lemma \ref{ker}. Indeed, let $v=(v_{\bold e})_{\bold e\in \bold E}\in Ker(\mathcal E_0) \cap L^2_{n}(\mathcal G)$, then
\begin{equation}\label{KerE_0}
-v''_{\bold e}+ v_{\bold e}-2\phi_0v_{\bold e}=0,\quad \bold e\in \bold E.
\end{equation}
Then, $v_{\bold e}=c_{\bold e} \phi'_0$ for $e\in \bold E$ and so $v_{\bold e}(0-)=v_{\bold e}(0+)=0$. Now, since $v\in L^2_{n}(\mathcal G)$, we obtain for $\bold e\in \bold E_{-}$ that $v_{\bold e}=c_0 \phi'_0$ with $c_0=c_{\bold e} $, and  for $\bold e\in \bold E_{+}$ that $v_{\bold e}=c_1 \phi'_0$ with $c_1=c_{\bold e} $.
Then from \eqref{E_0} we obtain $nc_1\phi''_0(0)=nc_0\phi''_0(0)$. Therefore, $v=c_0 \Phi'_0$.

For item $2)$, we will used extension theory for symmetric operators. Indeed,  we consider   the $2n\times 2n$-diagonal matrix   operator
   \begin{equation}\label{F0}
 \mathcal F_0=diag\Big(-\frac{d^2}{dx^2},..., -\frac{d^2}{dx^2}\Big).
 \end{equation}
 with domain
  \begin{equation}\label{F1}
D(\mathcal F_0)= \{u\in H^2(\mathcal G) : u(0-)=u(0+)=0,  \sum_{\bold e \in \bold E_{+}} u'_{\bold e}(0+)-\sum_{\bold e \in \bold E_{-}} u'_{\bold e}(0-)=0\}.
\end{equation}
Then $( \mathcal F_0, D( \mathcal F_0))$ represents a closed symmetric operator densely defined on $L^2(\mathcal G)$ (we note that $\bigoplus\limits_{\bold e\in  \bold E_{-} }C_c ^{\infty}(-\infty, 0)  \oplus \bigoplus\limits_{\bold e\in  \bold E_{+} }C_c ^{\infty}(0, +\infty)\subset D(\mathcal F_0)$).  Moreover, the
adjoint operator $( \mathcal F_0^*, D( \mathcal F_0^*))$ is given by (see Proposition \ref{F_0} in Appendix below)
  \begin{equation}\label{F*}
 \mathcal F_0^*=\mathcal F_0, \quad D( \mathcal F_0^*)=\{u\in H^2(\mathcal G) : u\in \mathcal C\}.
 \end{equation}
Next, from \eqref{F*}, the deficiency indices  for $( \mathcal F_0, D( \mathcal F_0))$ are $n_{\pm}( \mathcal F_0)=1$. Then, from the Krein-von Neumann extension theory for symmetric operators (see \cite{Albe}, Theorem A.1) and from Proposition \ref{F_0} (Appendix) we obtain that all self-adjoint extension of $( \mathcal F_0, D( \mathcal F_0))$, denoted by $(\mathcal L_Z, D(\mathcal L_Z))$, can be parametrized by $Z\in \mathbb R$ as $\mathcal L_Z=\mathcal F_0$ and $u\in D(\mathcal L_Z)$ if and only if $u\in \mathcal C$ and $u$ satisfying \eqref{tail3}. Next, we define the following bounded operator on $L^2(\mathcal G)$
 \begin{equation*}
  \mathcal B_0=\left(\begin{array}{cc} M_{0,+}& 0 \\0 & M_{0,-} \end{array}\right),\quad M_{0,\pm}=\text{diag}\Big(1 -2\phi_{0},..., 1 -2\phi_{0} \Big)
 \end{equation*}
 with $M_{0,\pm}$
 being $n\times n$-diagonal matrices. Then, from \cite{Nai67}-Chapter IV-Theorem 6 follows that the symmetric operators $\mathcal F_0$ and   $\tilde{\mathcal F}_0=\mathcal F_0+  \mathcal B_0$ with $D(\tilde{\mathcal F}_0)=D(\mathcal F_0)$ have the same deficiency indices, $n_{\pm}(\tilde{\mathcal F}_0)=n_{\pm}(\mathcal F_0)=1$.  Thus $(\mathcal E_0, D(\mathcal E_0))$ belongs to the family of the 
self-adjoint extensions of $\tilde{\mathcal F}_0$.  
 
 Next we see that the symmetric operator $\tilde{\mathcal F}_0$ with domain  $D(\tilde{\mathcal F}_0)=D(\mathcal F_0)$ in \eqref{F1}, it  is non-negative. Indeed,  it is easy to verify that for $u=(u_{\bold e})_{\bold e\in \bold E} \in H^2(\mathcal G)$  the following identity holds 
\begin{equation}\label{identity'}
-u_{\bold e}''+ u_{\bold e}- 2\phi_{0}u_{\bold e}=-
\frac{1}{\phi'_{0}}\frac{d}{dx}\left[(\phi'_{0})^2\frac{d}{dx}\left(\frac{u_{\bold e}}{\phi'_{0}}\right)\right], 
\end{equation}
for $x<0$ if $\bold e\in \bold E_{-}$,  $x>0$ if $\bold e\in \bold E_{+}$. Using the above equality and integrating by parts, we get for $u=(u_{\bold e})_{\bold e\in \bold E} \in D(\tilde{\mathcal F}_0)$
\begin{equation}\label{nonneg1'}\begin{split}
\langle\tilde{\mathcal F}_0 u,u\rangle&=\sum_{\bold e \in \bold E_{-}}
\int\limits_{-\infty}^{0}(\phi'_{0})^2\left|\frac{d}{dx}\left(\frac{u_{\bold e}}{\phi'_{0}}\right)\right|^2dx +\sum_{\bold e \in \bold E_{+}} \int\limits_{0}^{+\infty}(\phi'_{0})^2\left|\frac{d}{dx}\left(\frac{u_{\bold e}}{\phi'_{0}}\right)\right|^2dx\\
&-\sum_{\bold e \in \bold E_{-}}\left[\frac{u_{\bold e }}{\phi'_{0}}\left[(\phi'_{0})^2\frac{d}{dx}\left(\frac{u_{\bold e }}{\phi'_{0}}\right)\right]\right]_{-\infty}^{0-} -\sum_{\bold e \in \bold E_{+}}\left[\frac{u_{\bold e }}{\phi'_{0}}\left[(\phi'_{0})^2\frac{d}{dx}\left(\frac{u_{\bold e }}{\phi'_{0}}\right)\right]\right]_{0+}^{+\infty}.
\end{split}
\end{equation}
The integral terms in \eqref{nonneg1'} are non-negative and equal zero if and only if $u\equiv 0$.  Due to the conditions $u(0-)=u(0+)=0$ and $\phi''_{0} (0\pm)\neq 0$, non-integral term vanishes, and we get $\tilde{\mathcal F}_0\geqq 0$.

Due to  Proposition \ref{semibounded} (Appendix),  we have that the  self-adjoint extension $\mathcal E_0$ of $\tilde{\mathcal F}_0$ satisfies $n(\mathcal E_0)\leqq 1$. Taking into account the notation $\Phi_0=(\phi_0)_{\bold e\in \bold E}$ for the solitary wave profile we have $\mathcal E_0 \Phi_0= \Psi$, 
with $ \Psi=(-\phi_0^2)_{\bold e\in \bold E}$ and so   $
\langle \mathcal E_0  \Phi_0, \Phi_0\rangle=-n\int\limits_{-\infty}^{0}\phi_{0,-}^3(x)dx- n\int\limits_{0}^{+\infty} \phi_{0, +}^3(x)dx<0$,
 then from minimax principle we  arrive at $n(\mathcal E_0)=1$. Moreover, since $\Phi_0=(\phi_0)_{\bold e\in \bold E}\in L^2_n(\mathcal G)$ we get $n(\mathcal E_0|_{L^2_{n}(\mathcal G)})=1$. 

Item $3)$ is an immediate consequence of Weyl's theorem (see Reed\&Simon \cite{RS}). This finishes the proof.
 \end{proof}
 
\begin{remark} We observe that, when we deal with deficiency indices, the operator $\mathcal E_0$ is assumed to act on complex-valued functions which however does not affect the analysis of negative spectrum of $\mathcal E_0$ acting on real-valued functions.
 \end{remark}
 
Combining Lemma \ref{Morse12} and the framework of the perturbation theory as in Lemma \ref{Morse} (see  \cite{AngGol17a}) we obtain the following Lemma. We note initially that for $u_{-}=(\phi_{-})_{\bold e\in \bold E_{-}}$,  $u_{+}=(\phi_{+})_{\bold e\in \bold E_{+}}$, it is not difficult to see the convergence $U_{Z} = (u_{-}, u_{+}) \to \Phi_0=(\phi_0)_{\bold e\in \bold E}$, as $Z\to 0$, in $H^1(\mathcal G)\cap L^2_n(\mathcal{G})$.

\begin{lemma}\label{eigenE_Z} There exist $Z_0>0$ and
two analytic functions  $\Theta : (-Z_0, Z_0)\to \mathbb
R$ and $\Upsilon: (-Z_0,Z_0)\to L^2_n(\mathcal{G})$ such
that
\begin{enumerate}
\item[$(i)$] $\Theta(0)=0$ and $\Upsilon
(0)=\Phi'_0$, where
$\Phi'_0=(\phi'_0)_{\bold e\in \bold E}$.

\item[$(ii)$] For all $Z\in (-Z_0, Z_0)$,
$\Theta(Z)$ is the simple isolated second eigenvalue of $\mathcal E_Z$ in $L^2_n(\mathcal{G})$, and $\Upsilon(Z)$ is the
associated eigenvector for $\Theta(Z)$.

\item[$(iii)$] $Z_0$ can be chosen small enough to ensure
that for $Z\in (-Z_0, Z_0)$ the spectrum of $\mathcal E_Z$ in $L^2_n(\mathcal{G})$ is positive, except at most
the first two eigenvalues.

\item[$(iv)$] Since $\lim_{Z\to 0}\langle NU_{Z}, \Pi(Z)\rangle=\|N\Phi_{0}\|^2>0$ we obtain that
\begin{equation}\label{S5a}
\langle NU_{Z}, \Pi(Z)\rangle \neq 0,
\end{equation}
at least for $Z$ small.  Thus, an continuation argument shows \eqref{S5a} for all $Z$.
\end{enumerate}
\end{lemma}

 By using the Taylor's theorem and by following a similar argument as in Proposition 3.9 in \cite{AngGol17a}  we establish how the perturbed second eigenvalue moves depending on the sign of $Z$. 
  
  \begin{proposition}\label{morseE_0} There exists
$0<Z_1<Z_0$ such that $\Theta(Z)>0$ for any
$Z\in (-Z_1,0)$, and $\Theta(Z)<0$ for any $Z\in
(0, Z_1)$. Thus, in the space $ L^2_n(\mathcal G)$ for
$Z$ small, we have $n(\mathcal E_Z)=1$ as $Z<0$,
and $n(\mathcal E_Z)=2$ as $Z>0$.
\end{proposition}

\begin{proof} [{\bf{Proof of Proposition  \ref{Morse11}}}]
  From Proposition \ref{morseE_0} we have for $Z$ small that  $n(\mathcal E_Z)=1$ as $Z<0$,
and $n(\mathcal E_Z)=2$ as $Z>0$. Thus for counting the Morse index of $\mathcal E_Z$ for any $Z$ we use a classical continuation argument based on the Riesz-projection as in step $vi)$-proof of Lemma \ref{Morse}- and Lemma \ref{Morse10}. This finishes the proof.
  \end{proof}

 The following lemma shows assumption $S_6)$. Similarly to the case of two half-lines, for  $\alpha_{+}=1$ and   $\omega=-\beta_{+}$, we have the differentiable family of stationary solutions a one-parameter $U_{Z,\omega}=(u_{-, \omega}, u_{+, \omega})$, with $u_{-, \omega}=(\phi_{-, \omega})$ and $u_{+, \omega}=(\phi_{+, \omega})$.  Thus,  for $\varphi_\omega= (- \frac{d}{d\omega} u_{-, \omega}, -\frac{d}{d\omega} u_{+, \omega})$  we have  $\varphi\equiv \varphi_\omega|_{\omega=1}\in D(\mathcal E_Z)$ and  $\mathcal E_Z \varphi=U_{Z}$. Thus with the former notation, we obtain the following result.

\begin{lemma}\label{assuS5} Let $Z\neq 0$.
   The smooth curve of profiles $\omega \in (\frac{Z^2}{4},+\infty)\to U_{Z,\omega}=(\phi_{-, \omega}, \phi_{+, \omega})\in D(\mathcal E_Z)\cap L^2_n(\mathcal G)$  satisfies for  $\varphi\equiv -  \frac{d}{d\omega} U_{Z,\omega}|_{\omega=1}$ the relations
 \begin{equation}\label{condi2}
 \mathcal E_Z \varphi=U_{Z},\quad\text{and},\quad \langle \varphi, U_{Z} \rangle<0.
 \end{equation} 
\end{lemma}

\begin{proof} [{\bf{Proof of Theorem \ref{main2}}}]
Let $Z>0$. From  Lemmas \ref{Morse10}-\ref{assuS5}, Proposition  \ref{Morse11}, relation \eqref{S5a}, subsection 4.2 and Theorem \ref{crit} we obtain the  linear instability property of the bump's profiles $U_{Z}$ for the KdV model \eqref{kdv3}. Let $Z<0$, then  from Lemmas \ref{Morse10}-\ref{assuS5} and Proposition  \ref{Morse11} we obtain via Theorem \ref{crit2} the  linear instability  of the tail's profiles $U_{Z}$. This finishes the proof.
  \end{proof}  

\begin{remark}\label{geral}
The extension of Theorem \ref{main2} for the general case of  the sequences $(\alpha_ \bold e)_{\bold e\in \bold E}=(\alpha_-, \alpha_+)$ and  $(\beta_ \bold e)_{\bold e\in \bold E}=(\beta_-, \beta_+) $, with $\alpha_-=\alpha_+$, $\beta_-=\beta_+$, can be obtained via the following steps: 

 \begin{enumerate}  
 
  \item[1)] Let $n\geqq 2$. For $\alpha_+=(\alpha_{1,+}, ..., \alpha_{n,+})$ and $\beta_+=(\beta_{1,+},...\beta_{n,+})$ we consider the associated either bumps or tails profiles  $U_{Z, \alpha_+, \beta_+}=(u_{-}, u_{+})$, where for $\phi_{+, \alpha_{i}, \beta_{i}}$ defined by \eqref{soli6} and $ \phi_{-, \alpha_{i}, \beta_{i}}(x)=\phi_{+, \alpha_{i}, \beta_{i}}(-x)$ for $x<0$, we have
 \begin{equation}\label{u-}
  u_{-}=(\phi_{-, \alpha_{i}, \beta_{i}})_{1\leqq i\leqq n},\;\;u_{+}=(\phi_{+, \alpha_{i}, \beta_{i}})_{1\leqq i\leqq n}.
  \end{equation}
In other words, we have $n$-profiles of either bump or tail type as in Figures 2 and 3 on the balanced graph $\mathcal G$. We note that {\it a priori} they do not need to be continuous at the graph-vertex. Thus, we obtain that $U_{Z, \alpha_+, \beta_+}=(u_{-}, u_{+})\in D(H_Z)\cap \mathcal C\subset D(\mathcal E_{Z,\alpha_+,  \beta_+ } )=\mathcal C \cap D_{Z, \alpha_+, \delta}$ if and only if 
   $$
  \beta_{1,+}+\frac{Z^2}{4} \alpha_{1,+}=\beta_{2,+}+\frac{Z^2}{4} \alpha_{2,+}=...=  \beta_{n,+}+\frac{Z^2}{4} \alpha_{n,+},\;\;\;\text{and}\;\;\;
 \sum_{i=1}^n  \alpha_{i,+}=n.
 $$ 
Here,  the $2n\times 2n$-matrix self-adjoint Schr\"odinger operator 
  \begin{equation}\label{tail1}
 \mathcal E_{Z,\alpha_+,  \beta_+ } =\left(\begin{array}{cc} \mathcal L_{Z, -}& 0 \\0 &  \mathcal L_{Z, +} \end{array}\right)
 \end{equation}
 are defined by the $n\times n$-diagonal matrices,
 \begin{equation}\label{tail2}  
  \mathcal L_{Z,\pm}=\text{diag}\Big(-\alpha_{1,+}\frac{d^2}{dx^2}- \beta_{1,+} -2\phi_{\pm, \alpha_{1}, \beta_{1}},..., -\alpha_{n,+}\frac{d^2}{dx^2}- \beta_{n,+} -2\phi_{\pm, \alpha_{n}, \beta_{n} } \Big),
\end{equation} 
and $D_{Z, \alpha_+, \delta}\subset H^2(\mathcal{G})$ is defined by
 \begin{equation}\label{gera} 
u\in D_{Z, \alpha_+, \delta} \Leftrightarrow  u(0-)=u(0+), \;\;\sum_{i=1}^n \alpha_{i,+}u'_{i, +}(0+)- \sum_{i=1}^n \alpha_{i,+} u'_{i, -}(0-)=Znu_{1, +}(0+).
 \end{equation}

  \item[2)] $ \mathcal E_{Z,\alpha_+,  \beta_+ }: D(\mathcal E_{Z,\alpha_+,  \beta_+ } )\to L^2(\mathcal G)$ is invertible: In fact, let $u=(u_{1,-},\cdot\cdot\cdot, u_{n,-}, u_{1,+},\cdot\cdot\cdot, u_{n,+})\in D(\mathcal E_{Z,\alpha_+,  \beta_+ } )$ and $\mathcal E_{Z,\alpha_+,  \beta_+ }u=0$. Then, $u_{i,\pm}=a_{i, \pm}\phi'_{\pm, i} $, $i=1,...,n$, and $\phi_{\pm, i} \equiv \phi_{\pm, \alpha_{i}, \beta_{i} }$. Then, since $\phi'_{+, i}(0+)=-\phi'_{-, i}(0-)$ and $u\in \mathcal C$ we obtain 
  $$
a_{i, +}=-  a_{i, -},\;\; a_{1, +}=...=a_{n, +},\;\; a_{1, -}=...=a_{n, -},\;\; i=1,...,n.
$$
  Next, since $u\in D_{Z, \alpha_+, \delta}$ and $U_{Z, \alpha_+, \beta_+}\in \mathcal C$ we obtain from \eqref{soli1} the relation 
  $$
na_{1,+}\frac{Z^2}{2} \phi_{+, 1}(0+)=2 a_{1,+} \sum_{i=1}^n \alpha_i\phi''_{+, i}(0+)=2 a_{1,+} \phi_{+, 1}(0+)\sum_{i=1}^n[ -\beta_{i,+} -\phi_{+, 1}(0+)].
$$ 
Suppose $ a_{1,+}\neq 0$. Then, we have the following chain of equality
\begin{equation}
\begin{split}
n\frac{Z^2}{4}&=- \sum_{i=1}^n [\beta_{i, +} -\frac32 (\beta_{i,+}+\frac{Z^2}{4} \alpha_{i,+})]=\frac12\sum_{i=1}^n [\beta_{i, +}+ \frac{Z^2}{4} \alpha_{i,+}] +n\frac{Z^2}{4}.
\end{split}
\end{equation}
Therefore, $\beta_{1, +}+ \frac{Z^2}{4} \alpha_{1,+}=0$ and so we obtain a contradiction because of $-\beta_{1, +}> \frac{Z^2}{4} \alpha_{1,+}$. Hence, $a_{1,+}=0$ and therefore $u\equiv 0$.

  \item[3)] The relations $n(\mathcal E_{Z,\alpha_+,  \beta_+ })=2$ for $Z>0$ and $n(\mathcal E_{Z,\alpha_+,  \beta_+ })=1$ for $Z<0$ follow via a perturbation analysis as in Proposition \ref{morseE_0}. Moreover, from \eqref{u-} and \eqref{soli6}  we obtain for $\psi= (\frac{d}{d\beta_i} u_{-},  \frac{d}{d\beta_i} u_{+})$ that $\mathcal E_{Z,\alpha_+,  \beta_+ }\psi=U_{Z,\alpha_+,  \beta_+ }$ and $\langle \psi, U_{Z,\alpha_+,  \beta_+}\rangle<0$.
   \end{enumerate}  

Therefore, from items 1)-2)-3) above and from Theorems \ref{crit}-\ref{crit2} we obtain the linear instability property of the solitons profiles $U_{Z,\alpha_+,  \beta_+}$ for every $Z\neq 0$.
\end{remark}

 \section{Appendix}
 
 Next, for convenience of the reader and because  of non-standard results  used in the body of the manuscript we formulate the following  results of the extension theory (see \cite{Nai67}). The first one reads as follows.
 
\begin{theorem}\label{d5} (von-Neumann decomposition)
Let $A$ be a closed, symmetric operator, then
\begin{equation}\label{d6}
D(A^*)=D(A)\oplus\mathcal N_{-i} \oplus\mathcal N_{+i}.
\end{equation}
with $\mathcal N_{\pm i}=Ker(A^*\mp iI)$. Therefore, for $u\in D(A^*)$ and $u=x+y+z\in D(A)\oplus\mathcal N_{-i} \oplus\mathcal N_{+i}$,
\begin{equation}\label{d6a}
A^*u=Ax+(-i)y+iz.
\end{equation}
\end{theorem}

\begin{remark} The direct sum in (\ref{d6}) is not necessarily orthogonal.
\end{remark}
Our second one result of the extension theory of symmetric operators give us a strategy for estimating the Morse-index of the self-adjoint extensions.

\begin{proposition}\label{semibounded}
Let $A$  be a densely defined lower semi-bounded symmetric operator (that is, $A\geq mI$)  with finite deficiency indices $n_{\pm}(A)=k<\infty$  in the Hilbert space $\mathcal{H}$, and let $\widetilde{A}$ be a self-adjoint extension of $A$.  Then the spectrum of $\widetilde{A}$  in $(-\infty, m)$ is discrete and  consists of at most $k$  eigenvalues counting multiplicities.
\end{proposition}

The following result was used in the proof of Lemma \ref{Morse12}.

\begin{proposition}\label{F_0}
It consider the closed symmetric operator densely defined on $L^2(\mathcal G)$, $( \mathcal F_0, D( \mathcal F_0))$, by \eqref{F0}-\eqref{F1}. 
Then, the deficiency indices   are $n_{\pm}( \mathcal F_0)=1$. Therefore, we have that all the self-adjoint extensions of $( \mathcal F_0, D( \mathcal F_0))$ can be parametrized by $Z\in \mathbb R$, namely, $(\mathcal L_Z, D(\mathcal L_Z))$, with the action $\mathcal L_Z\equiv \mathcal F_0$ and $u\in D(\mathcal L_Z)$ if and only if $u\in  \mathcal C\cap D_{Z,\delta}$ (see \eqref{conti}-\eqref{tail3}).
\end{proposition}

\begin{proof}
We show initially  that the adjoint operator $( \mathcal F_0^*, D( \mathcal F_0^*))$ of $( \mathcal F_0, D( \mathcal F_0))$   is given by 
  \begin{equation}\label{F*2}
 \mathcal F_0^*=\mathcal F_0, \quad D( \mathcal F_0^*)=\{u\in H^2(\mathcal G) : u\in \mathcal C\}.
 \end{equation}
Indeed, formally for $u, v\in H^2(\mathcal G)$ we have 
\begin{equation}\label{relaself}
\begin{split}
 \langle\mathcal F_0v, u\rangle
 &=-\sum_{\bold e\in \bold E_{-}}v'_{\bold e}(0)u_{\bold e}(0)+\sum_{\bold e\in \bold E_{-}}v_{\bold e}(0)u'_{\bold e}(0) +
 \sum_{\bold e\in \bold E_{+}}v'_{\bold e}(0)u_{\bold e}(0)-\sum_{\bold e\in \bold E_{+}}v_{\bold e}(0)u'_{\bold e}(0)\\
 &\quad \quad +\langle v,\mathcal F_0u\rangle.
 \end{split}
 \end{equation}
Denote by $D_0^*=\{u\in H^2(\mathcal G) : u\in \mathcal C\}$. Then we will show $D_0^*=D( \mathcal F_0^*)$. Indeed, we see initially $D_0^*\subset D( \mathcal F_0^*)$. So,  for $u\in D_0^*$ and $v\in D( \mathcal F_0)$ follow from \eqref{relaself}
\begin{align}\label{relaself2}
 \langle\mathcal F_0v, u\rangle
 =u_{1,+}(0+)[-\sum_{\bold e\in \bold E_{-}}v'_{\bold e}(0)+
 \sum_{\bold e\in \bold E_{+}}v'_{\bold e}(0)]+\langle v,\mathcal F_0u\rangle=\langle v,\mathcal F_0u\rangle
 \end{align}
then $u\in D(\mathcal F_0^*)$ and $\mathcal F_0^*u=\mathcal F_0u$.

 Let us show the inverse inclusion $D_0^*\supseteq
D(\mathcal F_0^*)$. Take $u\in D(\mathcal F_0^*)$, then
for any $v\in D(\mathcal F_0)$ we have from \eqref{relaself}
\begin{align}\label{relaself3}
 \langle\mathcal F_0v, u\rangle
 =-\sum_{\bold e\in \bold E_{-}}v'_{\bold e}(0)u_{\bold e}(0)+
 \sum_{\bold e\in \bold E_{+}}v'_{\bold e}(0)u_{\bold e}(0)+\langle v,\mathcal F_0u\rangle=\langle v,\mathcal F^*_0u\rangle=\langle v,\mathcal F_0u\rangle.
 \end{align}
Thus, we arrive for any $v\in D(\mathcal F_0)$ at the equality 
\begin{equation}\label{adjoint}
 \sum_{\bold e\in \bold E_{+}}v'_{\bold e}(0)u_{\bold e}(0)-\sum_{\bold e\in \bold E_{-}}v'_{\bold e}(0)u_{\bold e}(0)=0
\end{equation} 
Next, it consider $v=(v_{1,-}, v_{2,-},..., v_{n,-}, ...0, 0,...0)\in D(\mathcal F_0)$ with $ v'_{3,-}(0-)=\cdot\cdot\cdot=v'_{n,-}(0-)=0$ and $ v'_{1,-}(0-)\neq 0$. Then from \eqref{adjoint} we obtain $v'_{1,-}(0-)[u_{1,-}(0-)-u_{2,-}(0-)]=0$ and so $u_{1,-}(0-)=u_{2,-}(0-)$. Repeating  similar arguments for $v$ being now $ v'_{4,-}(0-)=\cdot\cdot\cdot=v'_{n,-}(0-)=0$ and $ v'_{3,-}(0-)\neq 0$ we get $u_{1,-}(0-)=u_{2,-}(0-)=u_{3,-}(0-)$ and so on. Finally
taking $v$ such that
${v}'_{n, -}(0-)= 0$, we arrive at $u_{1,-}(0-)=u_{2,-}(0-)=u_{3,-}(0-)=\cdot\cdot\cdot=u_{n-1,-}(0-)$, and
consequently $u_{1,-}(0-)=u_{2,-}(0-)=u_{3,-}(0-)=\cdot\cdot\cdot=u_{n,-}(0-)$. Similarly, we see $u_{1,+}(0+)=u_{2,+}(0+)=u_{3,+}(0+)=\cdot\cdot\cdot=u_{n,+}(0+)$. Lastly, we see that $u_{1,-}(0-)=u_{1,+}(0-)$. Thus, let $v\in D(\mathcal F_0)$ such that $v'_{2,-}(0-)=\cdot\cdot\cdot=v'_{n,-}(0-)=v'_{2,+}(0+)=\cdot\cdot\cdot=v'_{n,+}(0+)=0$ and   $v'_{1,+}(0+)\neq 0$. Then from \eqref{adjoint} and from the relation $v'_{1,+}(0+)[u_{1,+}(0+)-u_{1,-}(0-)]=0$ follow that $u\in D^*_0$. Therefore, \eqref{F*2} holds.

From \eqref{F*2} we obtain that the deficiency indices  for $( \mathcal F_0, D( \mathcal F_0))$ is $n_{\pm}( \mathcal F_0)=1$. Indeed,  $Ker(\mathcal F_0^*\pm iI)=[\Psi_{\pm}]$ with $\Psi_{\pm}=(\Psi_{\bold e, \pm})_{\bold e\in \bold E}$ defined by
\begin{equation}
\Psi_{\bold e, \pm}=\left\{ \begin{array}{ll}
\Big(\frac{i}{k_{\pm}}e^{\mp ik_{\pm}x},...,\frac{i}{k_{\pm}}e^{\mp ik_{\pm}x}\Big),\;\;x<0,\;\; \bold e \in  \bold E_{-}\\
\\
\Big(\frac{i}{k_{\pm}}e^{\pm ik_{\pm}x},...,\frac{i}{k_{\pm}}e^{\pm ik_{\pm}x}\Big),\;\;x>0,\;\; \bold e \in  \bold E_{+}
 \end{array}  \right.
\end{equation}
$k^2_{\pm}=\mp i$, $Im(k_{-})<0$ and $Im(k_{+})>0$.
 
 Next, let us show that the domain of any self-adjoint extension
$\widehat{\mathcal F}$ of the operator $\mathcal F_0$ in \eqref{F0} and domain \eqref{F1} (and acting on complex-valued functions) is given by $D_{Z, \delta}$ in \eqref{tail3}. Indeed, we recall that $D(\widehat{\mathcal F})$ is a restriction of $D({\mathcal F}^*_0)$, so $ D(\widehat{\mathcal F})\subset \mathcal C$, moreover,  due to von-Neumann decomposition above and \cite[Theorem A.1]{Albe} follow
$$
D(\widehat{\mathcal F})=\left\{u\in H^2(\mathcal G): u= u_0+c \Psi_{-}+ce^{i\theta}\Psi_{+}:\,  u_0\in \mathcal F_0, c\in\mathbb{C},\theta\in[0,2\pi)\right\},
$$
Thus, it is easily seen that for $u\in D(\widehat{\mathcal F})$, we
have
\begin{align}\label{Zcondi}
\sum_{\bold e\in \bold E_{+}}u'_{\bold e}(0+)- \sum_{\bold e\in \bold E_{-}}u'_{\bold e}(0-)=2cn(1-e^{i\theta}),\;\; u_{1,+}(0+)=-c(e^{i\frac{\pi}{4}}-e^{i(\theta-\frac{\pi}{4})}).
\end{align}
From the last equalities it follows that  
\begin{align}\label{Zcondi2}
\sum_{\bold e\in \bold E_{+}}u'_{\bold e}(0+)- \sum_{\bold e\in \bold E_{-}}u'_{\bold e}(0-)=Znu_{1,+}(0+),\,\,
\text{where}\,\,
Z=\frac{-2(1-e^{i\theta})}{e^{i\frac{\pi}{4}}-e^{i(\theta-\frac{\pi}{4})}}\in\mathbb{R}\cup \{\pm \infty\}.
\end{align}
This finishes the proof.
\end{proof}


The idea of the following results  is to establish initially  a representation formula for the unitary group   associated to the linear evolution equation
\begin{equation}\label{group1}
\left\{ \begin{array}{ll}
u_t=A_{Z} u,\quad t\in \mathbb R\\
u(0)= u_0\in D(A_{Z}),
  \end{array}  \right.
\end{equation}
where $(A_Z, D(A_Z))$ is determined in Proposition \ref{L} (the case of two half-lines).  After that we establish the corresponding formula  in the case of the operators  $(H_Z, D(H_Z))$  determined in \eqref{L_Z} (the case of a balanced star graph). Thus, 
without loss of generality we assume $\alpha_-=\alpha_+=\beta_-=\beta_+=1$. Since $A_Z$ is a skew-self-adjoint operator, by Stone's theorem, the solution $u(t)=W(t)u_0$ is given by a  unitary group $\{W(t)\}_{t\in \mathbb R}$ on $L^2(\mathcal G)$ with associated infinitesimal generator  $A_Z$. Thus, for denoting $W(t)=W_{-}(t)\oplus W_{+}(t)$ and $w=(p,q)\in L^2(\mathcal G)=L^2(-\infty, +0)\oplus L^2(0, +\infty)$ we can see the action of $W(t)$ on $w$ as 
$$
W(t)w\equiv (W_{-}(t)p, W_{+}(t)q).
$$
The purpose of the following results is to establish explicit formulas for every $W_{\pm}$.

 \begin{lemma}\label{group2} Let $q\in L^2(0, +\infty)$ and $Re \lambda>0$. 
The non-homogeneous linear problem
 \begin{equation}\label{group3}
\left\{ \begin{array}{ll}
 & Nv(x)=q(x),\quad 0\leqq x<+\infty\\
 \\
 &v(0)=a_0, \; v(x), v'(x)\to 0\;\;\text{as}\;\; x\to+\infty.
  \end{array}  \right.
 \end{equation}  
 with $Nv(x)\equiv \lambda v(x)+v'(x)+v'''(x)$,  it has the representation
  \begin{equation}\label{group4}
 v_{+}(x)= a_0e^{\gamma_1x}+\int_0^\infty G_{+}(x, \zeta, \lambda)q(\zeta)d\zeta, \quad x\geqq 0
 \end{equation} 
where $G_{+}=G_{+}(x, \zeta, \lambda)$ is the associated Green's function for \eqref{group3} and $Re \gamma_1<0$.
\end{lemma}
 
 \begin{proof} Consideration is first directed to find the Green's function $G_{+}$ associated to the non-homogeneous linear problem \eqref{group3}, namely, with $a_0=0$ and $q\equiv 0$. Indeed, let $\gamma_1, \gamma_2, \gamma_3$ being the three roots of the characteristic equation
  \begin{equation}\label{group5}
\lambda +\gamma +\gamma^3=0, \quad \text{for}\;\; Re \lambda>0,
\end{equation} 
ordered so that 
\begin{equation}\label{roots}
Re \gamma_1<0,\;\; Re \gamma_2>0,\;\; Re\gamma_3>0.
\end{equation}
As we know $G_{+}$ is given as the unique solution of the problem
\begin{equation}\label{group6}
\left\{ \begin{array}{ll}
 & Ng(x,\zeta)=\delta(x-\zeta),\quad 0< x, \zeta <+\infty,\\
 \\
 &\text{for}\; 0<x<\zeta\;\text{we have}\; g(0)=0,\;\;\text{and}\;\; g(+\infty)=g'(+\infty)=0,\\
 \\
& g, \frac{dg}{dx}\;\text{are continuous at}\; x=\zeta;\quad\frac{d^2g}{dx^2}\Big|_{x=\zeta^+}-\frac{d^2g}{dx^2}\Big|_{x=\zeta^-}=1.
\end{array}  \right.
 \end{equation}  
Thus, since  the equation $Nv(x)=0$, for $x>0$,  has  the following fundamental set of solutions $\{e^{\gamma_1x}, e^{\gamma_2x}, e^{\gamma_3x}\}$ we obtain that the conditions $g(+\infty)=g'(+\infty)=0$ imply
\begin{equation}\label{group7}
g(x, \zeta)=d(\zeta) e^{\gamma_1x},\;\;\text{for}\; \zeta<x<+\infty.
 \end{equation} 
Next, the condition $g(0,\zeta)=0$ implies
\begin{equation}\label{group8}
g(x, \zeta)= a(\zeta)(e^{\gamma_2x}-e^{\gamma_3x})+ b(\zeta)(e^{\gamma_1x}-e^{\gamma_2x})
,\;\;\text{for}\;0<x<\zeta.
 \end{equation} 
Then, from the  conditions of continuity and jump for $g$ we obtain after an application of  Kramer's rule that
\begin{equation}\label{group8a}
\begin{array}{ll}
a(\zeta)=\frac{-1}{\Delta(\lambda)} &(\gamma_2-\gamma_1)e^{-\gamma_3\zeta}, \;\;b(\zeta)=\frac{-1}{\Delta(\lambda)} \Big[(\gamma_2-\gamma_1)e^{-\gamma_3\zeta}+ (\gamma_1-\gamma_3)e^{-\gamma_2\zeta}\Big],\\
\\ 
\text{and},& d(\zeta)=\frac{-1}{\Delta(\lambda)} \Big[(\gamma_2-\gamma_1)e^{- \gamma_3\zeta}+ (\gamma_1-\gamma_3)e^{-\gamma_2\zeta} +  (\gamma_3-\gamma_2)e^{-\gamma_1\zeta} \Big],\
\end{array}
\end{equation} 
where
$$
\Delta(\lambda)=(\gamma_1-\gamma_2)(\gamma_1-\gamma_3)(\gamma_2-\gamma_3).
$$
Therefore, $G_{+}$ is given explicitly by
\begin{equation}\label{group9}
\begin{array}{lll}
G_{+}(x, s, \lambda)&=\frac{1}{\Delta(\lambda)} \Big [(\gamma_3-\gamma_1)e^{\gamma_1x-\gamma_2\zeta}+(\gamma_1-\gamma_2)e^{\gamma_1x-\gamma_3\zeta}\\
\\
&+Y(x,\zeta)(\gamma_2-\gamma_3)e^{\gamma_1(x-\zeta)}\\
\\
&(1-Y(x,\zeta))\Big((\gamma_1-\gamma_3)e^{\gamma_2(x-\zeta)}+(\gamma_2-\gamma_1)e^{\gamma_3(x-\zeta)}\Big)\Big]
\end{array}
\end{equation} 
and 
\begin{equation}\label{group10}
Y(x,\zeta)=\left\{ \begin{array}{ll}
 & 1\quad\text{if}\;\; 0\leqq \zeta\leqq x,\\
 \\
 &0\quad\text{otherwise}.
 \end{array}  \right.
 \end{equation}  
Then, the solution for \eqref{group3} is given immediately by the superposition principle as the formula in \eqref{group4}.

  \end{proof}

Next, we find   the part $W_{-}$ of $W$ in \eqref{group3}.

 \begin{lemma}\label{group11} Let $p\in L^2(-\infty, 0)$ and $Re \lambda>0$. 
The non-homogeneous linear problem
 \begin{equation}\label{group12}
\left\{ \begin{array}{ll}
 & Nv(x)=p(x),\quad -\infty<x<0\\
 \\
 &v(0-)=a_1, v'(0-)=a_2, \; v(x)\to 0\;\;\text{as}\;\; x\to-\infty.
  \end{array}  \right.
 \end{equation}  
 with $Nv(x)\equiv \lambda v(x)+v'(x)+v'''(x)$,  it has the representation
  \begin{equation}\label{group13}
 v_{-}(x)=\alpha_1e^{\gamma_2x}+ \alpha_2e^{\gamma_3x}+\int_{-\infty}^0 G_{-}(x, \zeta, \lambda)p(\zeta)d\zeta, \quad x\leqq 0
 \end{equation} 
where $G_{-}=G_{-}(x, \zeta, \lambda)$ is the associated Green's function for \eqref{group12} and $Re \gamma_2>0$, $Re \gamma_3>0$. The constants $\alpha_i$ are chosen such that $ v_{-}(0-)=a_1$ and 
$ v'_{-}(0-)=a_2$
\end{lemma}
 
 \begin{proof} We start  by finding the Green's function $G_{-}$ associated to the non-homogeneous linear problem \eqref{group3}, namely, with $a_1=a_2=0$ and $p\equiv 0$. Indeed, let $\gamma_1, \gamma_2, \gamma_3$ being the three roots of the characteristic equation \eqref{group5} such that $
Re \gamma_1<0$, $Re \gamma_2>0$, $Re\gamma_3>0$.
As we know $G_{-}$ is given as the unique solution of the problem
\begin{equation}\label{group14}
\left\{ \begin{array}{ll}
 & Ng(x,\zeta)=\delta(x-\zeta),\quad -\infty< x, \zeta <0,\\
 \\
 &\text{for}\; \zeta<x<0\;\text{we have}\; g(0-)=g'(0-)=0,\;\;\text{and}\;\;  g(-\infty)=0,\\
 \\
& g, \frac{dg}{dx}\;\text{are continuous at}\; x=\zeta;\quad\frac{d^2g}{dx^2}\Big|_{x=\zeta^+}-\frac{d^2g}{dx^2}\Big|_{x=\zeta^-}=1.
\end{array}  \right.
 \end{equation}  
Thus, since  the equation $Nv(x)=0$, for $x<0$,  has  the following fundamental set of solutions $\{e^{\gamma_1x}, e^{\gamma_2x}, e^{x\gamma_3x}\}$ we obtain that  condition $g(+\infty)=0$ implies
\begin{equation}\label{group15}
g(x, \zeta)=r(\zeta) e^{\gamma_2x}+s(\zeta) e^{\gamma_3x},\;\;\text{for}\; -\infty<x<\zeta<0.
 \end{equation} 
Next, the condition $g(0-,\zeta)=g'(0-,\zeta)=0$ implies
\begin{equation}\label{group16}
g(x, \zeta)= t(\zeta)\Big((\gamma_3-\gamma_2)e^{\gamma_1x}-(\gamma_3-\gamma_1)e^{\gamma_2x}+(\gamma_2-\gamma_1)e^{\gamma_3x}\Big)
,\;\;\text{for}\;\zeta<x<0.
 \end{equation} 
Then, from the  conditions of continuity and jump for $g$ we obtain 
\begin{equation}\label{group17}
\begin{array}{lll}
r(\zeta)=\frac{1}{\Delta(\lambda)} &\Big [(\gamma_3-\gamma_1)e^{-\gamma_2\zeta}+(\gamma_1-\gamma_3)e^{-\gamma_1\zeta}\Big],
\;\;t(\zeta)=\frac{1}{\Delta(\lambda)} e^{-\gamma_1\zeta},\\
\\ 
\text{and},& s(\zeta)=\frac{1}{\Delta(\lambda)} \Big[(\gamma_1-\gamma_2)e^{- \gamma_3\zeta}+ (\gamma_2-\gamma_1)e^{-\gamma_1\zeta} \Big],\
\end{array}
\end{equation} 
where
$\Delta(\lambda)=(\gamma_1-\gamma_2)(\gamma_1-\gamma_3)(\gamma_2-\gamma_3)$.
Therefore, $G_{-}$ is given by
\begin{equation}\label{group18}
\begin{array}{ll}
G_{-}(x, s, \lambda)&=\frac{1}{\Delta(\lambda)} \Big [(\gamma_1-\gamma_3)e^{\gamma_2x-\gamma_1\zeta}+(\gamma_2-\gamma_1)e^{\gamma_3x-\gamma_1\zeta}\\
\\
&+Y_{-}(x,\zeta)(\gamma_3-\gamma_2)e^{\gamma_1(x-\zeta)}\\
\\
&(1-Y_{-}(x,\zeta))\Big((\gamma_3-\gamma_1)e^{\gamma_2(x-\zeta)}+(\gamma_1-\gamma_2)e^{\gamma_3(x-\zeta)}\Big)\Big]
\end{array}
\end{equation} 
and 
\begin{equation}\label{group19}
Y_{-}(x,\zeta)=\left\{ \begin{array}{ll}
 & 1\quad\text{if}\;\; \zeta\leqq x\leqq 0,\\
 \\
 &0\quad\text{otherwise}.
 \end{array}  \right.
 \end{equation}  
Then, the solution for \eqref{group12} is given via the superposition principle by the formula in \eqref{group13}.
\end{proof}

Next, we determine the resolvent operator for the skew-self-adjoint operator $(A_Z, D(A_Z))$ in \eqref{domain8}.

\begin{proposition}\label{group20} Let $\lambda \in \mathbb C$ such that $Re \lambda>0$, $\alpha_{-}=\alpha_{+}=\beta_{-}=\beta_{+}=1$ and $Z\in \mathbb R$. Then the resolvent operator for $A_Z$, $R(\lambda; A_Z)=(\lambda I-A_Z)^{-1}:L^2(\mathcal G)\to D(A_Z)$  has the representation
 for  $\omega=(p,q)\in L^2(-\infty, 0)\oplus L^2(0, +\infty)$  as 
 $$
 R(\lambda; A_Z)\omega=(R_{-}(\lambda; A_Z)p, R_{+}(\lambda; A_Z)q)=(v_{-}, v_{+})
 $$
  with $v_{\pm}$ defined by \eqref{group4} and  \eqref{group13}, respectively. The constants $a_0, \alpha_1, \alpha_2$ in \eqref{group4}-\eqref{group13} are uniquely determined by the condition $(v_{-}, v_{+})\in D(A_Z)$.
  \end{proposition}
 
 \begin{proof} Let $\omega=(p,q)\in L^2(-\infty, 0)\oplus L^2(0, +\infty)$ and $v=(v_{-}, v_{+})=R(\lambda; A_Z)\omega$. Then we obtain that $v_{-}, v_{+}$ satisfy the system
 \begin{equation}\label{group21}
\left\{ \begin{array}{lll}
 & \lambda v_{-}(x)+v_{-}'(x)+v_{-}'''(x)=p(x),\quad -\infty<x<0\\
 \\
 & \lambda v_{+}(x)+v_{+}'(x)+v_{+}'''(x)=q(x),\;\;\;\;\quad 0<x<+\infty\\
 \\
 &v_{-}(0-)=v_{+}(0+), \;\; v'_{+}(0+)-v'_{-}(0-)=Zv_{-}(0-)\\
 \\
 &v''_{+}(0+)-v''_{-}(0-)=\frac{Z^2}{2}v_{-}(0-)+Zv_{-}(0-),
  \end{array}  \right.
 \end{equation}   
 and therefore $v_{+}, v_{-}$ are defined by the formulas in \eqref{group4} and \eqref{group13}, respectively. The constants $a_0, \alpha_1, \alpha_2$ in \eqref{group4}-\eqref{group13} are the unique solution for the system
 \begin{equation}\label{group22}
  \mathcal A\left(\begin{array}{c}a_0 \\ \alpha_1 \\ \alpha_2\end{array}\right)=\left(\begin{array}{c}0 \\-\frac{d}{dx}\int_0^\infty G_{+}(x, \zeta, \lambda)q(\zeta)d\zeta\Big |_{x=0^{+}} \\ -\frac{d^2}{dx^2}\int_{-\infty}^0 G_{-}(x, \zeta, \lambda)p(\zeta)d\zeta\Big |_{x=0^{-}} -\frac{d^2}{dx^2}\int_0^\infty G_{+}(x, \zeta, \lambda)q(\zeta)d\zeta\Big |_{x=0^{+}}
 \end{array}\right)
\end{equation}
with
 \begin{equation}\label{group23}
 \mathcal A=\left(\begin{array}{ccc}1& -1 & -1\\ \gamma_1 & -(\gamma_2+Z) &  -(\gamma_3+Z) 
 \\ \gamma_1^2 & -(\gamma_2^2+\frac{Z^2}{2}+Z\gamma_2) & -(\gamma_3^2+\frac{Z^2}{2}+Z\gamma_3)\end{array}\right).
 \end{equation} 
We note that $det(\mathcal A)=(\gamma_3-\gamma_2)\Big [ \frac{Z^2}{2}+Z(\gamma_2+\gamma_3-\gamma_1) +(\gamma_1-\gamma_2)(\gamma_1-\gamma_3)\Big]\neq 0$ for all $Z\in \mathbb R$,  because of the Girard's relations
\begin{equation} 
\gamma_1+\gamma_2+\gamma_3=0,\;\;\gamma_1\gamma_2+\gamma_1\gamma_2+\gamma_2\gamma_3=1,\;\;\gamma_1\gamma_2\gamma_3=-\lambda,
 \end{equation} 
imply that the second-degree polynomial equation $ \frac{Z^2}{2}+Z(\gamma_2+\gamma_3-\gamma_1) +(\gamma_1-\gamma_2)(\gamma_1-\gamma_3)=0$ does not have real roots.
This finishes the proof.
\end{proof}

\begin{proposition}\label{group24}  The unitary group $\{W(t)\}_{t\in \mathbb R}$ associated to equation \eqref{group1} can be written for  $\omega=(p,q)\in L^2(-\infty, 0)\oplus L^2(0, +\infty)$ as $
W(t)w\equiv (W_{-}(t)p, W_{+}(t)q)$, with
 \begin{equation}\label{group25}
 \begin{array}{ll}
 &W_{-}(t)p(x)=\frac{1}{2\pi i}\int_{r-i\infty}^{r+i\infty} e^{\lambda t}R_{-}(\lambda; A_Z)p(x)d\lambda,\quad x\leqq 0,\\
 \\
 &W_{+}(t)q(x)=\frac{1}{2\pi i}\int_{r-i\infty}^{r+i\infty} e^{\lambda t}R_{+}(\lambda; A_Z)q(x)d\lambda,\quad x\geqq 0,
 \end{array} 
 \end{equation}   
with
 \begin{equation}\label{group26}
 \begin{array}{ll}
R_{-}(\lambda; A_Z)p(x) =\alpha_1e^{\gamma_2x}+ \alpha_2e^{\gamma_3x}+\int_{-\infty}^0 G_{-}(x, \zeta, \lambda)p(\zeta)d\zeta, \quad x\leqq 0,\\
\\
R_{+}(\lambda; A_Z)q(x) =\alpha_3e^{\gamma_1x}+\int_0^{+\infty}G_{+}(x, \zeta, \lambda)q(\zeta)d\zeta, \qquad\qquad\quad x\geqq 0,
\end{array} 
\end{equation} 
 where $G_{\pm}(x, \zeta, \lambda)$ are the associated Green's functions for \eqref{group3} and  \eqref{group12}, respectively, and $\alpha_1, \alpha_2, \alpha_3\in \mathbb R$ are uniquely determined by the condition $(R_{-}(\lambda; A_Z)p, R_{+}(\lambda; A_Z)q)\in D(A_Z)$.
\end{proposition}
 
 \begin{proof}
Using the Laplace transform and Proposition \ref{group20}, it follows from semi-group theory that for  $\omega=(p,q)\in L^2(-\infty, 0)\oplus L^2(0, +\infty)$
 \begin{equation}\label{group27}
  \begin{array}{ll}
W(t)w(x)&=\frac{1}{2\pi i}\int_{r-i\infty}^{r+i\infty} e^{\lambda t}R(\lambda; A_Z)w(x)d\lambda\\
\\
&=\frac{1}{2\pi i}\Big( \int_{r-i\infty}^{r+i\infty} e^{\lambda t}R_{-}(\lambda; A_Z)p(x)d\lambda, \int_{r-i\infty}^{r+i\infty} e^{\lambda t}R_{+}(\lambda; A_Z)q(x)d\lambda\Big).
\end{array} 
\end{equation} 
This finishes the proof.
 \end{proof}

\begin{remark}\label{roots2}
	We note  by using Girard's relations that the three roots of the equation $\lambda-\gamma+\gamma^3=0$ for Re $\lambda>0$ also can be ordered as in \eqref{roots}. Thus Proposition \ref{group24} is also valid on the case $\alpha_-=\alpha_+=1$ and $\beta_-=\beta_+=-1$.
\end{remark}

The next basic result about the invariance of the subspace $D(H_Z)\cap \mathcal C$ (defined in Section 6, \eqref{L_Z}-\eqref{conti}) by the unitary group generated by $H_Z$ was used in the proof of the instability Theorem \ref{main2} in the case of a balanced  star graph.  We note that in the case of two half-lines, this invariance property for the domain $D(A_Z)$ in \eqref{domain8} is obvious, but for general star graphs  is not immediate.

\begin{proposition}\label{group28}  Consider the skew-self-adjoint operator $(H_Z, D(H_Z))$ in \eqref{L_Z} on a star graph $\mathcal G$ with a structure $\bold E\equiv \bold E_{-}\cup \bold E_{+}$ where $|\bold E_{+}|=|\bold E_{-}|=n$, $n\geqq 2$. Let $\{W(t)\}_{t\in \mathbb R}$ be the unitary group associated to $H_Z$. Then, for $\mathcal C$ defined by
  \begin{equation}\label{group29} 
 \mathcal C =\{(u_ \bold e)_{\bold e\in \bold E}\in L^2(\mathcal G): u_{1, -}(0-)=...=u_{n, -}(0-)=u_{1, +}(0+)=...=u_{n, +}(0+)\},
  \end{equation}
we have that $D(H_Z)\cap \mathcal C$ is invariant by the group $W(t)$.
\end{proposition}

\begin{proof} From Proposition \ref{group24}, for $(u_ \bold e)_{\bold e\in \bold E}=(u_{1, -}, ...,u_{n, -}, u_{1, +},...,u_{n, +})\in L^2(\mathcal G)$ we define 
 \begin{equation}\label{group30}
 \begin{array}{ll}
 &W_{j,-}(t)u_{j, -}(x)=\frac{1}{2\pi i}\int_{r-i\infty}^{r+i\infty} e^{\lambda t}R_{-}(\lambda; H_Z)u_{j, -}(x)d\lambda,\quad x\leqq 0,\;\; 1\leqq j\leqq n,\\
 \\
 &W_{j,+}(t)u_{j, +}(x)=\frac{1}{2\pi i}\int_{r-i\infty}^{r+i\infty} e^{\lambda t}R_{+}(\lambda; H_Z)u_{j, +}((x)d\lambda,\quad x\geqq 0,\;\; 1\leqq j\leqq n,
 \end{array} 
 \end{equation}   
with $R_{\pm}(\lambda; H_Z)$-components given by \eqref{group26}. Thus, we can write
$$
W(t)=\bigoplus\limits_{j=1}^nW_{j,-}  \oplus \bigoplus\limits_{j=1}^nW_{j,+}.
$$
Now, for $u=(u_ \bold e)_{\bold e\in \bold E}\in D(H_Z)$ is obvious by definition of $W_{j,\pm}$ that $W(t)u\in  D(H_Z)$ (see Proposition \ref{group20}). Moreover, for all $j$, 
$$W_{j,-}(t)u_{j, -}(0-)=W_{j,+}(t)u_{j, +}(0+),\quad \text{for all}\;\;j\in \{1,2,\cdot\cdot\cdot, n\}.$$
Thus, for $u=(u_ \bold e)_{\bold e\in \bold E}\in D(H_Z)\cap \mathcal C$ follows immediate from \eqref{group30}  that $W(t)u\in \mathcal C$. This finishes the proof.

\end{proof}

\begin{remark} From Remark \ref{roots2} follows immediately that Proposition \ref{group28} is also true for $(\alpha_e)_{e\in \bold E}=(1)_{e\in \bold E}$ and $(\beta_e)_{e\in \bold E}=(-1)_{e\in \bold E}$.
	\end{remark}

\vskip0.2in

 \noindent
{\bf Acknowledgements.} 
 J. Angulo  was supported 
in part by   CNPq/Brazil Grant. M. Cavalcante 
 wishes to thank the University of S\~ao Paulo, where part of the
paper was written, for the financial support, for the invitation and hospitality. The authors would like to thank O. Lopes for fruitful conversations about his manuscript \cite{Lopes}.

\end{document}